\documentclass[a4paper,11pt]{article}
\usepackage{amsfonts}
\usepackage{amssymb}
\usepackage{amsmath}
\usepackage{amsthm}
\usepackage{multirow}
\usepackage{diagbox}
\usepackage{amsopn}
\usepackage{gensymb} 
\usepackage{fancyhdr}
\usepackage{graphicx}
\usepackage{mathrsfs}
\usepackage{latexsym}
\usepackage{graphicx}
\usepackage{subfigure}
\usepackage{color}
 \usepackage{calc}
  \usepackage{tikz}
  
\usepackage{setspace}
\usepackage{textcomp} 

\numberwithin{equation}{section} \allowdisplaybreaks
\theoremstyle{plain}
\newtheorem{thm}{Theorem}[section]

\theoremstyle{remark}
\newtheorem{remark}{Remark}
\numberwithin{equation}{section}
\newtheorem{Lemma}[thm]{Lemma}
\newtheorem{prop}[thm]{Proposition}
\newtheorem{assumption}{Assumption}

\newtheorem{corollary}[thm]{Corollary}

\begin{document}
\title{\bf{Q-Hermite polynomials chaos approximation of likelihood function based on q-Gaussian prior in Bayesian inversion}}

\author{\hspace{-2.5cm}{\small Zhi-Liang Deng$^{1}$ Xiao-Mei Yang$^2$\thanks{
Corresponding author: yangxiaomath@163.com; yangxiaomath@home.swjtu.edu.cn
Supported by NSFC No. 11601067, 11771068 and No.11501087, the Fundamental Research Funds for the Central Universities No. 2682018ZT25 and ZYGX2018J085.
},} \\
\hspace{-1.5cm}{\scriptsize $1.$ School of Mathematical Sciences,  University of Electronic Science and Technology of China,
Chengdu 610054, China}\\
\hspace{-2.5cm}{\scriptsize $2.$ School of Mathematics,
Southwest Jiaotong University,
Chengdu 610031, China}
}
\date{}
\maketitle

\begin{abstract}
\noindent In real applications, the construction of prior and acceleration of sampling for posterior are usually two key points of Bayesian inversion algorithm for engineers. In this paper, q-analogy of Gaussian distribution, q-Gaussian distribution, is introduced as the prior of inverse problems.  And an acceleration algorithm based on spectral likelihood approximation is discussed. We mainly focus on the convergence of the posterior distribution in the sense of Kullback-Leibler divergence
 when approximated likelihood function and truncated prior distribution are used. Moreover, the convergence in the sense of total variation and Hellinger metric is obtained. 
In the end two numerical examples are displayed.

\noindent \textbf{Key words:} q-Gaussian prior; q-Hermite polynomials; Spectral likelihood approximation; Kullback-Leibler divergence

\noindent \textbf{MSC 2010}: 65R32, 65R20
\end{abstract}

\section{Introduction}
Bayesian approach has been widely applied to inverse models and parameter inference models \cite{dashti,kaipio,stuart,stuart1,tarantola,marzouk1,marzouk2,yan1,yan2,yan3,yang}. It provides a handy framework for the data analysis in the real engineering problems \cite{beck,hadidi,yuen} based on Bayes theorem,
\begin{align}\label{in1.1}
\mu^y(dx)=\frac{f(y| x)\mu(dx)}{\int f(y| x)\mu(dx)},
\end{align} 
where the distribution $\mu$ characterizes prior
knowledge about the unknown parameter $x$ and $f(y| x)$ determines the likelihood function. In real applications, two key aspects of Bayesian method perplex engineers and researchers: the prior distribution and acceleration of simulation. Firstly, the prior information is coded before obtaining the measured data. This means that one needs to have the first understanding for the unknown parameter and make some survey according to their own experience and actual situation. However, this is usually a challenging task in some real problems, e.g., reservoir, non-destructive inspection, CT etc. This prompts us to explore broader and more suitable  prior distributions for specific problems. In \cite{stuart,stuart1}, some prior distributions, e.g., Gauss, uniform, Besov prior, have been discussed for ill-posed operator equations. We will introduce q-Gauss distribution, q-analogy of Gauss distribution, into the research of inverse problems in this paper. In what follows, more information of q-Gauss distribution will be provided.
 On the other hand, in the Bayesian framework parameters are frequently estimated by Markov chain Monte Carlo(MCMC) sampling techniques which typically have slow convergence. In fact, MCMC methods evaluate sequentially the posterior probability density at many different points in the parameter space, in which the forward model needs to be solved  for each sample parameter to determine the likelihood function. This requires a computationally intensive undertaking (e.g., the solution of a system of PDEs).
%
 Therefore, numerical acceleration algorithm is key for real applications. A kind of important acceleration methods is to reduce computational cost in solving a statistical inverse problems \cite{frangos}: reducing the cost of forward simulations, reducing the dimension of the input space and reducing the number of samples. 

Q-Gaussian distribution, as an analogy to Gaussian distribution, has been discussed by many authors (see \cite{bryc,bozejko1,bozejko2,leeuwen} and references there), and widely used in quantum physics \cite{simon1,simon2}. In classical probability, the central limit theorem shows that the standardized sum of $n$ classically independent identifically distributed random variables converges to a Gaussian random variable as $n$ goes to infinity. This process depends on the commutative notion of independence. However, this conventional commutative relation is unsuitable for some real applications. Some necessary extensions need to be done.
In \cite{bozejko1},  Bo\.{z}jko and Speicher generalize the commutative independence in a deformation of Brownian motion by introducing a parameter $q\in [-1, 1]$. When $q=1$, it is the classical case, $q=-1$  the anti-commutative independence and $q=0$ the free independence. These commutative notions are used to characterize some quantum physics phenomena \cite{maassen,meyer,speicher,voiculescu}. 
The density function of q-Gaussian distribution is represented by an infinite series and this truncation error of the partial sum series is discussed in \cite{szablowski}. 
 Compared with classical Gaussian random variables, q-Gaussian variables for $-1<q<1$ are bounded, with which we can depict some bounded physical parameters, e.g., the diffusion coefficients in heat conduction problems, the order of fractional diffusion equation.  In addition, with big $q$ (greater than some constant $q_0$), the density functions of q-Gaussian distribution are unimodal. While $q$ is small, they are bimodal. Bimodal probability distributions have important applications in economic, natural problems. For more discussions about q-Gaussian distribution, one can refer to \cite{bozejko1,bozejko2,szablowski}.


Numerical acceleration has  been always concerned by scientists and engineers.  As stated above, in statistical inference problems, the main acceleration ideas include improving the sampling efficiency, reducing the dimensionality of input parameter and reducing the evaluation cost of the forward problem. In improving sampling efficiency, one can see \cite{christen,cui,higdon,efendiev}. In \cite{cui}, Cui et al integrate the reduced-order model construction process into an adaptive MCMC algorithm, in which the reduced-order model is used to increase the efficiency of MCMC sampling.
For reducing the dimensionality, ad hoc method is to expand the unknown parameter in its Karhunen-Lo\'{e}ve expansion according to the given prior knowledge \cite{dashti,stuart} and truncate the expansion into the  partial sum. The  expansion coefficients  of the truncation series are viewed as the substitute of the unknown. 
For reducing the evaluation cost of the forward problem, one usually tries to transform complex forward models into a simplified or coarsed version, e.g., model reduction method, or construct an approximation or 'surrogate' of the forward problem. A lot of research has been devoted to these fields, for instance,  some model reduction and surrogate based approaches \cite{arridge,frangos,galbally,jin,lieberman,manzoni}, generalized polynomial chaos (gPC) methods \cite{marzouk1,marzouk2,marzouk3,xiu1,xiu2,yan1,yan2} and Gaussian process regression method \cite{kennedy,rasmussen,stuart2}. 
 Recently, on the basis of surrogate method of forward model, some authors propose a 'more-direct' surrogate algorithm, spectral likelihood approximation method \cite{nagel}. This approach does not replace the forward model directly, but replace the likelihood function with the orthogonal polynomials expansion. By this approach, the polynomial chaos expansion (PCE) has clearer explanation in mathematics.
 In this paper, we consider a spectral likelihood approximation approach based on q-Hermite polynomials, which are orthogonal with q-Gaussian distribution weight. 

This paper is organized by the following: In Section 2, we introduce some basic 
knowledge about q-Gaussian distribution and q-Hermite polynomials. We give a convergence rate for the truncation q-Hermite polynomial expansion. In Section 3, the Bayesian inversion based on q-Gaussian prior is stated. 
 In Section 4, we consider polynomial chaos expansion of likelihood function based on q-Hermite polynomials.
 In Section 5, we analyze the Kullback-Leibler divergence in two approximation process: the likelihood and the prior approximation.
Two numerical examples are given in Section 6.






\section{Preliminaries}
In this section, we first give some basic 
conceptions and notations and  then analyze the convergence rate for truncated q-Hermite polynomial expansion. We just discuss 1-dimensional case in this section. For multi-dimensional case, it is a direct extension. 

Denote for $n\in \mathbb{N}_0$ and $-1<q<1$
\begin{align}\label{qh2.1}
&[n]_q:=\frac{1-q^n}{1-q}=1+q+\cdots+q^{n-1},\,\, [0]_q:=0,\\
&(a; q)_n=\prod_{k=0}^{n-1}(1-aq^k).
\end{align}
The density function $f^{(q)}(x)$ of q-Gaussian \cite{bozejko2} is supported by the interval $[-\frac{2}{\sqrt{1-q}}, \frac{2}{\sqrt{1-q}}]$, on which
\begin{align}\label{qha2.3}
f^{(q)}(x)=\frac{1}{\pi}\sqrt{1-q}\sin\theta\prod\limits_{n=1}^\infty (1-q^n)|1-q^ne^{2i\theta}|^2
\end{align}
with $x=\frac{2}{\sqrt{1-q}}\cos\theta$,  $\theta\in(0, \pi)$ and $i=\sqrt{-1}$. 
The density function $f^{(q)}(x)$ has the following expansion and the truncated error estimation \cite{szablowski} :
\begin{Lemma}\label{lemma2.1}
For $-1<q<1$, one has for $x\in [-\frac{2}{\sqrt{1-q}}, \frac{2}{\sqrt{1-q}}]$
\begin{align}\label{qh_a2.6}
f^{(q)}(x)=\frac{\sqrt{1-q}}{2\pi}\sqrt{4-(1-q)x^2}\sum_{k=1}^\infty
(-1)^{k-1}q^{\left(\begin{array}{c}k \\2\end{array}\right)}T_{2k-2}
(\frac{x\sqrt{1-q}}{2}),
\end{align}
where $T_k(x)$ is Chebyshev polynomial of  the second kind defined by
\begin{align*}
T_k(x)=\frac{\sin((k+1)\arccos x)}{\sqrt{1-x^2}}.
\end{align*}
Denote
\begin{align}\label{qh_a2.7}
f^{(q)}_J(x):=\frac{\sqrt{1-q}}{2\pi}\sqrt{4-(1-q)x^2}
\sum_{k=1}^{J-1}
(-1)^{k-1}q^{\left(\begin{array}{c}k \\2\end{array}\right)}T_{2k-2}
(\frac{x\sqrt{1-q}}{2}).
\end{align}
Moreover for $J\geq 4$, it has the following estimation
\begin{align}\label{qh_a2.8}
\sup_{|x|<2/\sqrt{1-q}}|f^{(q)}(x)-f_J^{(q)}(x)|\leq \frac{|q|^{(J-1)(J-2)/2}}{\pi(1-q^2)^2}.
\end{align}

\end{Lemma}

Let $\mathcal{I}_q:=(\tilde{x}-\frac{2}{\sqrt{1-q}}, \tilde{x}+\frac{2}{\sqrt{1-q}})$ and $\mathcal{L}^{2}_{\mu_q}:=\mathcal{L}^{2}_{\mu_q}(\mathcal{I}_q)$ be the Hilbert space of functions that are square integrable with respect to the measure 
\begin{align}\label{pc4.1}
\mu_q(dx):=\frac{1}{\sqrt{\Xi}}f^{(q)}(\frac{x-\tilde{x}}{\sqrt{\Xi}})dx.
\end{align}
To simplify, we set $\tilde{x}=0, \Xi=1$. 
The inner product $(\cdot, \cdot)_{\mathcal{L}^{2}_{\mu_q}}$ and norm $\|\cdot\|_{\mathcal{L}^{2}_{\mu_q}}$ are defined by
\begin{align}\label{pc4.2}
&(\psi_1, \psi_2)_{\mathcal{L}^{2}_{\mu_q}}=\int_{\mathcal{I}_q}\psi_1(x)\psi_2(x)\mu_q(dx), \,\, \forall \psi_1, \psi_2\in \mathcal{L}^{2}_{\mu_q},\\
&\label{pc4.3} \|\psi\|_{\mathcal{L}^{2}_{\mu_q}}=\sqrt{\int_{\mathcal{I}_q}|\psi(x)|^2\mu_q(dx)},\,\, \forall \psi\in\mathcal{L}^{2}_{\mu_q}.
\end{align}

 
Q-Hermite polynomials \cite{koekoek} are determined by the following  recurrence relation
\begin{align}\label{qh2.2}
xH_n^{(q)}(x)=H_{n+1}^{(q)}(x)+[n]_qH_{n-1}^{(q)}(x),\,\, n\geq 1
\end{align}
with $H_0^{(q)}(x)=1$ and $H_1^{(q)}(x)=x$. 
 They are orthogonal to each other with respect to  measure \eqref{pc4.1}.
We can write the orthogonal relation in the following
\begin{align}\label{qh2.4}
\int_{-\frac{2}{\sqrt{1-q}}}^{\frac{2}{\sqrt{1-q}}}H_n^{(q)}(x) H_m^{(q)}(x)\mu_q(dx)=\delta_{mn}[n]_{q}!,
\end{align}
where $[n]_{q}!:=[1]_q\cdots[n]_q$ and $\delta_{mn}$ is the Kronecker delta.

Define the q-differential  $D_q, 0<|q|<1$
 \cite{koekoek}:
\begin{align}
D_qf(x)=\frac{\delta f(x)}{\delta x}, \,\, x=\cos\theta,
\end{align}
where 
\begin{align}
&\delta f(e^{i\theta})=f(q^{\frac{1}{2}}e^{i\theta})-f(q^{-\frac{1}{2}}e^{i\theta}),\\
&\delta x=-\frac{1}{2}q^{-\frac{1}{2}}(1-q)(e^{i\theta}-e^{-i\theta}).
\end{align}
It follows for q-Hermite polynomials $H_n^{(q)}(x)$ that \cite{koekoek}
\begin{align}
D_qH_n^{(q)}(x)=q^{-\frac{n-1}{2}}[n]_q H_{n-1}^{(q)}(x).
\end{align}
and generally
\begin{align}\label{derivative2.12}
D_q^{(k)}H_n^{(q)}(x)&=\prod_{l=1}^{k}q^{-\frac{n-l}{2}}[n-l+1]_qH_{n-k}^{(q)}(x), \,\,k=1, 2, \cdots.
\end{align}
For $f\in \mathcal{L}^2_{\mu_q}$, it has the following expansion
\begin{align}\label{217f}
f(x)=\sum_{n=0}^\infty a_n H_n^{(q)}(x).
\end{align}
Denote the first $N+1$ terms sum of \eqref{217f} by
\begin{align}
f_N(x)=\sum_{n=0}^N a_n H_n^{(q)}(x).
\end{align}
Following a similar proof in \cite{augustin}, we have the following truncated error estimation.
\begin{prop}\label{theorem2.2}
Let $0<|q|<1$ and $k\geq 1$. For  $f\in \mathcal{L}^2_{\mu_q}$ being $k$ times continuously q-differentiable, the convergence rate 
\begin{align}\label{q_ex}
\|f-f_N\|^2_{\mathcal{L}^2_{\mu_q}}\leq\frac{|q|^{\frac{(2N-1-k)k}{2}}}{\prod\limits_{l=1}^{k}[N-l+2]_q}\|D_q^{(k)}f\|^2_{\mathcal{L}^2_{\mu_q}}
\end{align}
can be obtained. Especially, for $f(x) = H^{(q)}_{N+1}(x)$, when $k=1$, we have 
\begin{align}\label{eeeee}
\|f-\sum_{n=0}^{N}a_n H_n^{(q)}\|^2_{\mathcal{L}^2_{\mu_q}}=\frac{|q|^N}{[N+1]_q}\|D_q^{(1)}f\|^2_{\mathcal{L}^2_{\mu_q}}.
\end{align}
\end{prop}
\begin{proof}
 We assume $0<q<1$. For $-1<q<0$, the proof is exactly same. 

By the orthogonality of q-Hermite polynomials, for any $f\in \mathcal{L}^2_{\mu_q}$, it holds the Parseval identify
\begin{align}\nonumber
\|f\|^2_{\mathcal{L}^2_{\mu_q}}&=(\sum_{n=0}^\infty a_n H_n^{(q)}(\cdot), \sum_{k=0}^\infty a_k H_k^{(q)}(\cdot))_{\mathcal{L}^2_{\mu_q}}\\
&=\sum_{n=0}^\infty [n]_q!a_n^2.\label{parf}
\end{align}
Using the formulation \eqref{derivative2.12}, with some simple calculations
we get \begin{align}
&\|D_q^{(k)}f\|^2_{\mathcal{L}^2_{\mu_q}}\nonumber\\
&=(\sum_{i=k}^\infty a_i\prod_{l=1}^k q^{-\frac{i-l}{2}}[i-l+1]_qH_{i-k}^{(q)}, \sum_{j=k}^\infty a_j\prod_{l=1}^k q^{-\frac{j-l}{2}}[j-l+1]_qH_{j-k}^{(q)})_{\mathcal{L}^2_{\mu_q}}\nonumber\\
&=\sum_{i, j=k}^\infty a_ia_j\prod_{l=1}^k q^{-\frac{i-l}{2}}[i-l+1]_q(\prod_{l=1}^k q^{-\frac{j-l}{2}}[j-l+1]_q)(H_{i-k}^{(q)}, H_{j-k}^{(q)})_{\mathcal{L}^2_{\mu_q}}\nonumber\\
&=\sum_{j=k}^\infty a_j^2[j-k]_q!(\prod_{l=1}^k q^{-\frac{j-l}{2}}[j-l+1]_q)^2\nonumber\\
&=\sum_{j=k}^\infty a_j^2[j-k]_q! q^{-\frac{(2j-1-k)k}{2}}(\prod_{l=1}^k [j-l+1]_q)^2.\label{parsevald}
\end{align}
Using the above \eqref{parf} and \eqref{parsevald} we have
\begin{align*}
&\|f-f_N\|^2_{\mathcal{L}^2_{\mu_q}}=\sum_{n=N+1}^\infty a_n^2 [n]_q!
\nonumber\\
&=\sum_{n=N+1}^\infty a_n^2 [n-k]_q! \prod_{l=1}^{k}[n-l+1]_q\nonumber\\
&=\sum_{n=N+1}^\infty a_n^2 [n-k]_q! q^{-\frac{(2n-1-k)k}{2}}\prod_{l=1}^{k}[n-l+1]_q  q^{\frac{(2n-1-k)k}{2}}\\
&\leq \sum_{n=N+1}^\infty a_n^2 [n-k]_q! q^{-\frac{(2n-1-k)k}{2}}\prod_{l=1}^{k}[n-l+1]_q \frac{\prod\limits_{l=1}^{k}[n-l+1]_q}{\prod\limits_{l=1}^{k}[N-l+2]_q} q^{\frac{(2n-1-k)k}{2}}\nonumber\\
 &\leq \frac{q^{\frac{(2N-1-k)k}{2}}}{\prod\limits_{l=1}^{k}[N-l+2]_q}\|D_q^{(k)}f\|^2_{\mathcal{L}^2_{\mu_q}}.\nonumber
\end{align*}
This completes the proof of \eqref{q_ex}. \\
Next, for $f(x) = H^{(q)}_{N+1}(x)$, 
the q-Hermite expansion coefficients of $f(x)$ hold
\begin{align}
a_{n}=\left\{
\begin{aligned}
&1,& n= N+1, \\
&0, & \text{otherwise}.
\end{aligned}
\right.
\end{align}
The Parseval equality \eqref{parf} gives
\begin{align}
\|f-f_{N}\|^2_{\mathcal{L}^2_{\mu_q}}=[N+1]_q!.
\end{align}
The right can be rewritten as
\begin{align}
[N+1]_q!=\frac{q^{-N}[N]_q![N+1]_q^2}{[N+1]_q}q^N.
\end{align}
This together with equality \eqref{parsevald} yields \eqref{eeeee}.
\end{proof}

\begin{remark}
For $q=0$, $H_n^{(0)}(x)$ are the Chebyshev polynomials.  When $f$ has $k$  continuous derivatives, then $|f(x)-f_N(x)|=O(N^{-(k-1)}).$ One can refer \cite{gil} for more details.
\end{remark}
\begin{remark}
 The estimation \eqref{q_ex} is consistent with that in \cite{augustin}, i.e.,
\begin{align}
\|f-f_N\|^2_{\mathcal{L}^2_{\mu_1}}\leq \frac{1}{\prod\limits_{l=1}^{k}(N-l+2)}\|f^{(k)}\|^2_{\mathcal{L}^2_{\mu_1}},
\end{align}
where $\mathcal{L}^2_{\mu_1}$ is the square integrable space with Gaussian weight.
 It is obvious that when $|q|\,\, (0<|q|\leq 1)$ is smaller, the convergence of $f_N$ to $f$ in \eqref{q_ex} is faster and  the result is superior to the classical one.

\end{remark}

%

\section{Bayesian inversion based on q-Gaussian prior}

Inverse problems are to find $x$, an input to a mathematical model, from given observation $y$. We have an equation of the form
\begin{align}\label{qpr3.1}
y=\varphi(x^\dag)+\eta,
\end{align}
where $\eta$ is the data noise, $x^\dag$ is the true solution and $\varphi:\mathcal{X}\rightarrow \mathcal{Y}$ is the forward operator, $\mathcal{X}, \mathcal{Y}$ are Banach spaces. To simplify the discussion, we assume $\mathcal{X}, \mathcal{Y}$ to be finite dimensional spaces $\mathbb{R}^m, \mathbb{R}^n$.
The most common used method to solve \eqref{qpr3.1} is regularization techniques, e.g., Tikhonov regularization that searches the minimizer of the following minimization problem
\begin{align}
\label{qpr3.2}
\text{arg}\min\limits_{x\in\mathcal{X}} \|\varphi(x)-y\|_{\mathcal{Y}}^2+\|x-x_0\|_{\mathcal{X}}^2,
\end{align}
where $x_0$ is the prior guess.

In Bayesian inversion, inverse problem \eqref{qpr3.1} is restated within probability framework. In details, we consider $\mathcal{X}, \mathcal{Y}$ as  sample spaces. Let $X$ be a random variable in $\mathcal{X}$ and $x$ be a realization of $X$ (In subsequent, we denote random variables by capital letters and its realization by the corresponding lower letters.).   The forward operator $\varphi$ maps probability space $(\mathcal{X}, \mathfrak{F}_{\mathcal{X}}, \mu_{\mathcal{X}})$ to probability space $(\mathcal{Y}, \mathfrak{F}_{\mathcal{Y}}, \mu_{\mathcal{Y}})$. Here $\mathfrak{F}$ and $\mu$ with the subscripts $\mathcal{X}, \mathcal{Y}$ denote the Borel $\sigma$-algebras and the probability measures in the corresponding space respectively (We omit the subscripts under no any confusion). 
   Instead of finding an estimation $x$ from an observation $y$, we explore the probability distribution $\mu^y(dx):=\mu(dx|Y=y)$ of random variable $X$ given by $Y$. According to Bayes' formula \eqref{in1.1}, it is transformed to determine likelihood function and prior distribution. We take the distribution of noise $\eta$ as the Gaussian, i.e.,
\begin{align}\label{qpr3.3}
\eta\sim N(0, \Gamma),
\end{align}
where $\Gamma$ is the noise covariance matrix. In this case, the likelihood function can be written as
\begin{align}\label{qpr3.4}
f(y| x)&=\frac{1}{(2\pi)^\frac{n}{2}\sqrt{\text{det}(\Gamma)}}\exp(-\frac{\|\Gamma^{-\frac{1}{2}}(\varphi(x)-y)\|^2}{2})\nonumber\\
&:=\frac{1}{(2\pi)^\frac{n}{2}\sqrt{\text{det}(\Gamma)}}\exp(-\frac{\|\varphi(x)-y\|_{\Gamma}^2}{2}),
\end{align} 
where $\text{det}$ denotes the determinant. We assume that the components of the uncertain parameter vector $X=(X_1, X_2, \cdots, X_m)$ are independent random variables $X_i$. For each component, we set its prior density as
the q-Gaussian defined by \eqref{pc4.1} with mean $\tilde{x}_i$ and variance $\Xi_i$.
For multi-dimensional case, with a slight abuse of notation, we still denote the measure by $\mu_q(dx)$ for multi-dimensional case.
 Therefore, the posterior density satisfies
\begin{align}\label{qpr3.5}
f^y(x)\propto f(y| x)\prod\limits_{i=1}^m \frac{1}{\sqrt{\Xi_i}}f^{(q)}(\frac{x_i-\tilde{x}_i}{\sqrt{\Xi_i}}).
\end{align}
Maximizing the posterior probability is equivalent to minimizing the following function
\begin{align}\label{qpr3.6}
\frac{1}{2}\|\varphi(x)-y\|^2_\Gamma-\sum_{i=1}^m\log(f^{(q)}(\frac{x_i-\tilde{x}_i}{\sqrt{\Xi_i}})):=\Psi(x, y)+H(x),
\end{align}
where $\Psi$ is the negative log likelihood, also called the potential function.
By the i.i.d. of $X$, we get the Hessian matrix of $H(x)$ 
\begin{align*}
\text{Hess}(x)&=(\frac{\partial^2 H}{\partial x_i\partial x_j})_{m\times m}
\\&=\text{diag}(\frac{\partial^2 H}{\partial x_1^2}, \frac{\partial^2 H}{\partial x_2^2}, \cdots, \frac{\partial^2 H}{\partial x_m^2}),
\end{align*}
where
\begin{align*}
\frac{\partial^2 H}{\partial x_i^2}=\frac{(\frac{df^{(q)}(x_i)}{dx_i})^2-\frac{d^2f^{(q)}(x_i)}{dx_i^2}f^{(q)}(x_i)}{(f^{(q)}(x_i))^2},\,\, i=1, 2, \cdots, m.
\end{align*}

\begin{Lemma}\cite{szablowski}
$f^{(q)}(x)$ is bimodal for $q\in (-1, q_0)$, where $q_0\approx-0.107$ is the largest real root of the equation $\sum_{k=0}^{\infty}(2k+1)^2q^{\frac{k(k+1)}{2}}=0$.
\end{Lemma}
The fact that $f^{(q)}(x)$ for $q\geq q_0$ is unimodal implies that $\frac{\partial^2 H}{\partial x_i^2}\geq 0$ and therefore $H$ is positive semidefinite, which means that $H$ is a convex penalty term of function \eqref{qpr3.6}. Whereas, when $-1<q<q_0$, there exists three extreme points $-a, 0, a$. It is obvious that $\frac{df^{(q)}}{dx}(0)=\frac{df^{(q)}}{dx}(\pm a) =0$ and $\frac{d^2f^{(q)}}{dx^2}(0)>0$, $\frac{d^2f^{(q)}}{dx^2}(\pm a)<0$. Moreover, because $f^{(q)}(x)$ is probability density function, it satisfies that $f^{(q)}(x)\geq 0$. Thereby, we have $\frac{\partial^2 H}{\partial x^2}\mid_{x=0}<0$ and $\frac{\partial^2 H}{\partial x^2}\mid_{x=\pm a}>0$, which implies $H$ is a non-convex penalty function. For non-convex constraint, the function \eqref{qpr3.6} has multi local minima. This case is typically hard to solve and analyze in classical optimization framework.

\section{Polynomial chaos expansion of likelihood function based on q-Hermite polynomials}

\subsection{Spectral likelihood approximation}

We know that q-Hermite polynomials $H_n^{(q)}$ are orthogonal with respect to measure $\mu_q(dx)$. Define multivariate polynomials
\begin{align}\label{pc4.4}
\mathcal{H}_\alpha^{(q)}(x):=H_{\alpha_1}^{(q)}(x_1)H_{\alpha_2}^{(q)}(x_2)\cdots H_{\alpha_m}^{(q)}(x_m),
\end{align}
where $\alpha=(\alpha_1, \alpha_2, \cdots, \alpha_m)\in\mathbb{N}_0^m$. It is obvious that polynomials $\mathcal{H}_\alpha^{(q)}$ are orthogonal with respect to measure $\mu_q(dx)$ in \eqref{pc4.1} 
\begin{align}\label{pc4.5}
(\mathcal{H}_\alpha^{(q)}, \mathcal{H}_\beta^{(q)})_{\mathcal{L}^{2}_{\mu_q}}=\left\{
\begin{aligned}
&[\alpha_1]_q![\alpha_2]_q!\cdots[\alpha_m]_q!, && \alpha=\beta \\
&0, &&  \alpha\neq\beta.
\end{aligned}
\right.
\end{align}
Polynomials $\{\mathcal{H}_\alpha^{(q)}\}$ form a complete orthogonal system of $\mathcal{L}^{2}_{\mu_q}$. The likelihood function $f(y| x)$ is measurable in the prior measure ${\mu_q(dx)}$. It can be expanded in
\begin{align}\label{pc4.6}
f(y| x)=\sum_{\alpha\in \mathbb{N}_0^m}a_\alpha^y\mathcal{H}_\alpha^{(q)}(x),
\end{align}
where $a_\alpha^y$ is the Fourier coefficients depending on data $y$ defined by
\begin{align}\label{pc4.7}
a_\alpha^y=\frac{(f(y|\cdot), \mathcal{H}_\alpha^{(q)}(\cdot))_{\mathcal{L}^{2}_{\mu_q}}}{[n_1]_q![n_2]_q!\cdots[n_m]_q!}.
\end{align}
In real numerical implementation, the expansion in \eqref{pc4.6} will be truncated to a finite summation
\begin{align}\label{pc4.8}
f_{\Lambda_N}(y| x)=\sum_{\alpha\in \Lambda_N}a_\alpha^y\mathcal{H}_\alpha^{(q)}(x),
\end{align}
where $\Lambda_N$ is finite multi-indices set defined by
\begin{align}\label{pc4.9}
\Lambda_N:=\{\alpha\in\mathbb{N}_0^m: \|\alpha\|_1=\sum_{i=1}^{m}|\alpha_i|\leq N\}.
\end{align}
This truncation means that we need collect all multivariate polynomials with order $\|\alpha\|_1$ smaller than or equal to $N$. The total number \cite{nagel,yan1,yan2} is 
\begin{align}\label{pc4.10}
P=\left(\begin{array}{c}m+N  \\N \end{array}\right)=\frac{(m+N)!}{m!N!}.
\end{align}
A simple way to reduce the dimension is to limit the number of regressors relies on hyperbolic truncation sets. 
For $0<l<1$ a quasinorm is defined as 
$\|\alpha\|_l=(\sum_{i=1}^m |\alpha_i|^l)^{1/l}$.
The corresponding hyperbolic truncation
scheme is then given as $\varpi=\{\alpha\in\mathbb{N}^m\mid \|\alpha\|_l\leq N\}.$
The convergence in the mean-square sense  is indicated 
in \cite{box,nagel} 
\begin{align}\label{pc4.11}
\|f(y|\cdot)-f_{\Lambda_N}(y|\cdot)\|_{\mathcal{L}^{2}_{\mu_q}}^2=\mathbb{E}^{\mu_q}((f(y| x)-f_{\Lambda_N}(y| x))^2)=\sum\limits_{\alpha\in \mathbb{N}^m\backslash\Lambda_N} |a_\alpha^y|^2.
\end{align}
If $f(y| x)$ is $k$ times continuous q-differentiable about $x$,  Theorem \ref{theorem2.2} shows the mean square error $\|f(y|\cdot)-f_{\Lambda_N}(y|\cdot)\|_{\mathcal{L}^{2}_{\mu_q}}\rightarrow 0$.
One can refer to \cite{augustin,muhlpfordt} for truncation errors for polynomial chaos expansions.

\subsection{Christoffel least square}
For obtaining the expansion coefficients $a_\alpha^y$, we adopt the stochastic collocation algorithm \cite{marzouk2,yan1}. The collocation equation will be solved in the least-square  procedure.  For given data $y$, we need find the minimizer
\begin{align}\label{leas4.12}
a_\alpha^y:=\arg\min\frac{1}{J}\sum_{j=1}^J |f(y| x^{(j)})-f_{\Lambda_N}(y| x^{(j)})|^2,
\end{align}
where $x^{(j)}=(x^{(j)}_1, x^{(j)}_2, \cdots, x^{(j)}_n)$ is a sample of $X\sim\mu_q(dx)$, i.e., $x^{(j)}_i\sim f^{(q)}(x)dx$ and $J$ is the sample number.
In general, the sample number $J$ is greater than the number $P$ in \eqref{pc4.10}, which leads to an overdetermined linear system
\begin{align}\label{chris4.13}
A^*A a^y_\alpha=A^*b,
\end{align}
where $A=(A_{kj})=(\mathcal{H}_{\alpha^{(k)}}(x^{(j)}))$ is a $J\times P$ Vandermonde-like matrix, $a^y_\alpha=(a^y_{\alpha^{(1)}}, a^y_{\alpha^{(2)}}, \cdots, a^y_{\alpha^{(P)}})^T$ and
$b=(f(y| x^{(1)}), f(y| x^{(2)}), \cdots, f(y| x^{(J)}))^T$.
 The popular techniques for solving \eqref{leas4.12} include interpolatory approaches, compressive sampling or $l^1$ regularization and the least-squares $l^2$ regularization. A new approach, called Christoffel least-squares, is proposed in \cite{narayan}. This method takes $x^{(j)}$ to be i.i.d. from another distribution $\nu(dx)$. Let $v(x)$ be the density function of $\nu(dx)$. The support of $v$ contains the support of $f^{(q)}(x)$. Instead of least-squares problem \eqref{leas4.12}, we solve
\begin{align}\label{chr4.14}
a_\alpha^y=\arg\min \frac{1}{J}\sum_{j=1}^J\kappa_j|f(y| x^{(j)})-f_{\Lambda_N}(y| x^{(j)})|^2,
\end{align}
where $\kappa_j=\frac{f^{(q)}(x^{(j)})}{v(x^{(j)})}$. Obviously, the solution is defined by
\begin{align}\label{kres4.15}
a_\alpha^y=\arg\min\limits_{a\in\mathbb{R}^m}\|\sqrt{\mathcal{K}}\tilde{A}a-\sqrt{\mathcal{K}}\tilde{b}\|^2,
\end{align}
where $\mathcal{K}$ is a $J\times J$ diagonal matrix with entries $\mathcal{K}_{jj}=\kappa_j$ and $\tilde{A}, \tilde{b}$ are defined like in \eqref{chris4.13} by replacing the samples drawn from distribution $\nu$. 
The solution can be obtained by solving the normal equation of \eqref{kres4.15}
\begin{align}
(\tilde{A}^*\mathcal{K}\tilde{A})a_\alpha^y=\tilde{A}^*\mathcal{K}\tilde{b}.
\end{align}
The CLS algorithm choose the weights $\kappa_j$ to be quantities that scale each row of $\sqrt{\mathcal{K}}\tilde{A}$ to have $l^2$ norm equal to the constant $P$, i.e.,
\begin{align}
\kappa_j=\frac{P}{\sum_{\alpha\in\Lambda_N}\mathcal{H}_\alpha^2(x^{(j)})}.
\end{align}
The measure $\nu$ is called pluripotential equilibrium measure,  which has density function for our settings, i.e., the q-Gaussian density $f^{(q)}(x)$
\begin{align}
v(x)=\frac{\sqrt{1-q}}{\pi\sqrt{4-(1-q)x^2}}.
\end{align}
This is Chebyshev density corresponding to the arcsin measure.
The equilibrium measure for multi-dimensional case is the product of the univariate measure for i.i.d. case.

\section{Convergence analysis}
In this section, we analyze the error between the exact posterior measure
\begin{align}
\mu^y:=\frac{f(y| x)\mu_q(dx)}{\int f(y| x)\mu_q(dx)}:=\frac{f(y| x)\mu_q(dx)}{\gamma}
\end{align}
 and its approximation version 
\begin{align}
\tilde{\mu}^y_{NJ}:=\frac{f_{\Lambda_N}(y| x)\mu_{q}^J(dx)}{\int  f_{\Lambda_N}(y| x)\mu_{q}^J(dx)}:=\frac{f_{\Lambda_N}(y| x)\mu_{q}^J(dx)}{\gamma_{NJ}},
\end{align}
where $\mu_{q}^J(dx)$ is defined by replacing $f^{(q)}(x)$ in $\mu_q(dx)$ \eqref{pc4.1} with its truncated version $f^{(q)}_J(x)$. 
In addition, we introduce the measure
\begin{align}
\tilde{\mu}^y_{N}:=\frac{f_{\Lambda_N}(y| x)\mu_q(dx)}{\int f_{\Lambda_N}(y| x)\mu_q(dx)}:=\frac{f_{\Lambda_N}(y| x)\mu_q(dx)}{\gamma_{N}}.
\end{align}
We will focus on the Kullback-Leibler (KL) divergence \cite{lu1,lu2,marzouk1,sanz} (also known as the relative entropy) of $\nu$ with respect to $\mu$
 to measure the bound 
\begin{align}
D_{KL}(\nu||\mu):=\left\{
\begin{aligned}
&\int \nu\log\frac{\nu}{\mu}=\mathbb{E}^\nu\log(\frac{\nu}{\mu}), &  \nu\ll\mu \\
&\infty, &  \text{otherwise},
\end{aligned}
\right.
\end{align}
where $\nu\ll\mu$ means $\nu$ is absolutely continuous with respect to $\mu$.  If $\mu$ and $\nu$ are two measures on a $\sigma$-algebra $\mathfrak{F}_\mathcal{X}$ of subsets of $\mathcal{X}$, we say that $\nu$ is absolutely continuous with respect to $\mu$ if $\nu(A)=0$ for any $A\in \mathfrak{F}_\mathcal{X}$ such that $\mu(A)=0$. If the measure $\nu$ is finite, i.e. $\nu(X)<\infty$, the property $\nu\ll\mu$ is equivalent to the following stronger statement: for any $\epsilon>0$ there is a $\delta>0$ such that $\nu(A)<\epsilon$ for every $A$ with $\mu(A)<\delta$.
The Kullback-Leibler divergence is vital in measuring the loss of information when $\nu$ is instead of $\mu$ in information theory. As seen in \cite{sanz}, the Kullback-Leibler divergence is non-negative. But, since it may not meet the symmetric and the triangle inequality, it is not a metric on the space of probability measures. Nevertheless, we can use it to quantify the proximity of the measures $\nu$ and $\mu$ by virtue of some inequalities like
\begin{align}\label{errr5.5}
D_{\text{TV}}(\nu, \mu):=\sup\{|\nu(A)-\mu(A)|: A\in \mathcal{X}\}\leq D_{KL}(\nu||\mu)^{\frac{1}{2}}
\end{align}
and 
\begin{align}\label{errr5.6}
D_{\text{Hell}}^2(\nu, \mu):=\frac{1}{2}\int(\sqrt{\nu}-\sqrt{\mu})^2\leq\frac{1}{2}D_{KL}(\nu||\mu),
\end{align}
where $D_{\text{TV}}(\nu, \mu), D_{\text{Hell}}^2(\nu||\mu)$ are the so-called total variation metric and Hellinger metric respectivelly.

For convenience, we assume the noise covariance $\Gamma=\delta^2 I$. 
\begin{assumption}\cite{stuart,stuart1}\label{assumption1}
The forward operator $\varphi: \mathbb{R}^m\rightarrow\mathbb{R}^n$ satisfies the following: for every $\epsilon>0$ there exists $M=M(\epsilon)\in\mathbb{R}$ such that, for all $x$,
\begin{align*}
\|\varphi(x)\|_2\leq\exp(\epsilon\|x\|_2^2+M).
\end{align*}
\end{assumption}
\begin{Lemma}\cite{dashti}\label{lemma5.1}
Let $\varphi$ satisfy Assumption \ref{assumption1}.
Then $\Psi$ satisfies: 
For every $r>0$, there exists $L=L(r)>0$
such that for all $x$ and $y\in\mathbb{R}^m$
with $\max\{\|x\|_2, \|y\|_2\}<r$,
\begin{align*}
\Psi(x, y)\leq L(r).
\end{align*}
\end{Lemma}
By the above assumption and Lemma \ref{lemma5.1}, we have the following convergence result.

\begin{thm}\label{theorem5.2} Let $\varphi$ satisfy assumption \ref{assumption1} and
 $\delta\leq \frac{\exp(-\frac{L(r)}{n})}{\sqrt{2\pi}}$. It holds that
\begin{align}
D_{KL}(\tilde{\mu}^y_{N}||\mu^y)\leq(2\pi)^{\frac{n}{2}}\exp(L(r))\mathbb{E}^{\mu_q}((f(y| x)-f_{\Lambda_N}(y| x))^2)\rightarrow 0
\end{align}
as $N\rightarrow \infty$.
\end{thm}
\begin{proof}

According to the non-negative of the relative entropy, we get
\begin{align*}
&D_{KL}(\tilde{\mu}^y_{N}||\mu^y)\leq D_{KL}(\tilde{\mu}^y_{N}||\mu^y)+D_{KL}(\mu^y||\tilde{\mu}^y_{N})
\\
&=\int\tilde{\mu}^y_{N}\log\frac{\tilde{\mu}^y_{N}}{\mu^y}
+\int\mu^y\log\frac{\mu^y}{\tilde{\mu}^y_{N}}\\
&=\int(\tilde{\mu}^y_{N}-\mu^y)\log\frac{\tilde{\mu}^y_{N}}{\mu^y}\\
&=\int(\frac{f_{\Lambda_N}(y| x)\mu_q(dx)}{\gamma_{N}}-\frac{f(y| x)\mu_q(dx)}{\gamma}
)\log\frac{f_{\Lambda_N}(y| x)}{f(y| x)}\frac{\gamma}{\gamma_N}\\
&=\int\mu_q(dx)(\frac{f_{\Lambda_N}(y| x)}{\gamma_{N}}-\frac{f(y| x)}{\gamma}
)\log\frac{f_{\Lambda_N}(y| x)}{f(y| x)}\\
&=\frac{1}{\gamma\gamma_N}\int\mu_q(dx)(\gamma f_{\Lambda_N}(y| x)-\gamma_N f(y| x)
)\log\frac{f_{\Lambda_N}(y| x)}{f(y| x)}\\
&=\frac{1}{\gamma\gamma_N}\int\mu_q(dx)(\gamma f_{\Lambda_N}(y| x)-\gamma_N f_{\Lambda_N}(y| x)\\
&+\gamma_N f_{\Lambda_N}(y| x)-\gamma_N f(y| x)
)\log\frac{f_{\Lambda_N}(y| x)}{f(y| x)}\\
&=\frac{\gamma-\gamma_N}{\gamma\gamma_N}\int\mu_q(dx)f_{\Lambda_N}(y| x)\log\frac{f_{\Lambda_N}(y| x)}{f(y| x)}\\
&+\frac{1}{\gamma}\int\mu_q(dx)(f_{\Lambda_N}(y| x)-f(y| x))\log\frac{f_{\Lambda_N}(y| x)}{f(y| x)}\\
&:=I_1+I_2.
\end{align*}
Since $(\gamma-\gamma_N)\log\frac{f_{\Lambda_N}(y| x)}{f(y| x)}\leq 0$, 
we know $I_1\leq 0$, which yields
\begin{align}
D_{KL}(\tilde{\mu}^y_{N}||\mu^y)\leq I_2.
\end{align}
Here
\begin{align}
I_2=\frac{1}{\gamma}\int\mu_q(dx)(f_{\Lambda_N}(y| x)-f(y| x))\log\frac{f_{\Lambda_N}(y| x)}{f(y| x)}.
\end{align}
When $\delta\leq \frac{\exp(-\frac{L(r)}{n})}{\sqrt{2\pi}}$, we get by Lemma \ref{lemma5.1}
\begin{align}
&f(y| x)=\frac{1}{(2\pi)^{\frac{n}{2}}\delta^n}\exp(-\frac{\|\varphi(x)-y\|_2^2}{2\delta^2})
\\
&\geq \frac{1}{(2\pi)^{\frac{n}{2}}\delta^n}\exp(-L(r))\geq 1
\end{align}
for $x\in B(0, r)$.
When $f(y| x)\geq 1$, we have $f_{\Lambda_N}(y| x)\geq1$ almost everywhere. In fact, according to the convergence 
\begin{align}
\mathbb{E}^{\mu_q}(f_{\Lambda_N}(y| x)-f(y| x))^2\rightarrow 0, \,\, \text{as} \,\, N\rightarrow \infty,
\end{align}
we have for arbitrary $\epsilon>0$, there exists $\tilde{N}>0$, when $N>\tilde{N}$
\begin{align}
\int(f_{\Lambda_N}(y| x)-f(y| x))^2\mu_q(dx)<\epsilon.
\end{align}
Define $E_1:=\{x\mid f_{\Lambda_N}(y| x)<1\}$ and $E_2=E_1^c$. 
If $$\mathfrak{m}_{\mu_q(dx)}(E_1):=\int_{E_1}\mu_q(dx)>0,$$
then we have
\begin{align}\nonumber
&\epsilon>\int(f_{\Lambda_N}(y| x)-f(y| x))^2\mu_q(dx)
\\
&=\int_{E_1}+\int_{E_2}(f_{\Lambda_N}(y| x)-f(y| x))^2\mu_q(dx)\\
&\geq \int_{E_1}(f_{\Lambda_N}(y| x)-f(y| x))^2\mu_q(dx)>0.\nonumber
\end{align}
The arbitrary of $\epsilon$ leads to a contradiction. Therefore, we can assume $f_{\Lambda_N}(y| x)\geq 1$. Likewise, since likelihood function $f(y| x)$ is non-negative for any $x$, we also suppose $f_{\Lambda_N}(y| x)\geq 0$. 
 By this, it follows that
\begin{align}
\log\frac{f_{\Lambda_N}(y| x)}{f(y| x)}\leq f_{\Lambda_N}(y| x)-f(y| x),\,\, \text{for} \,\, x\in B(0, r).
\end{align}
For $I_2$, we have 
\begin{align*}
&I_2=\frac{1}{\gamma}\int\mu_q(dx)(f_{\Lambda_N}(y| x)-f(y| x))\log\frac{f_{\Lambda_N}(y| x)}{f(y| x)}\\
&\leq \frac{1}{\gamma}(\int_{\|x\|_2\leq r}+\int_{\mathbb{R}^n\backslash B(0, r)}) \mu_q(dx)(f_{\Lambda_N}(y| x)-f(y| x))\log\frac{f_{\Lambda_N}(y| x)}{f(y| x)}\\
&\leq \frac{1}{\gamma}\int_{\|x\|_2\leq r}\mu_q(dx)(f_{\Lambda_N}(y| x)-f(y| x))^2\\
&+\frac{1}{\gamma}\int_{\mathbb{R}^n\backslash B(0, r)}\mu_q(dx)(f_{\Lambda_N}(y| x)-f(y| x))\log\frac{f_{\Lambda_N}(y| x)}{f(y| x)}.
\end{align*}
So long as $r$ is sufficient large, we can guarantee $\mu_q(dx)=0$ for $x\in\mathbb{R}^n\backslash B(0, r)$. In fact, $r$ is greater than $\frac{2}{\sqrt{1-q}}$ that assure $B(0, r)\supset\mathcal{I}_q$ can meet this point. 
The boundedness of $\gamma$ from below  can be obtained by the similar reasons. 
\begin{align}
\gamma&=\int \mu_q(dx) f(y| x)\nonumber\\
&=\int_{\|x\|_2\leq r}\mu_q(dx) f(y| x)+\int_{\mathbb{R}^n\backslash B(0, r)}\mu_q(dx) f(y| x)\\
&\geq \frac{1}{(2\pi)^{\frac{n}{2}}\delta^n}\exp(-L(r)).\nonumber
\end{align}
Therefore, we have
\begin{align}
I_2\leq (2\pi)^{\frac{n}{2}}\exp(L(r))\mathbb{E}^{\mu_q}((f(y| x)-f_{\Lambda_N}(y| x))^2),
\end{align}
which yields
\begin{align}
D_{KL}(\tilde{\mu}^y_{N}||\mu^y)\leq(2\pi)^{\frac{n}{2}}\exp(L(r))\mathbb{E}^{\mu_q}((f(y| x)-f_{\Lambda_N}(y| x))^2).
\end{align}

\end{proof}


Moreover, using Lemma \ref{lemma2.1}, we get 
\begin{thm}\label{theorem5.3}
When $J\rightarrow\infty$, we have
\begin{align}
D_{KL}(\tilde{\mu}^y_{NJ}||\tilde{\mu}^y_{N})\rightarrow 0.
\end{align}
\end{thm}
\begin{proof}
We only analyze the case of $n=1$. For $n>1$, the independence of each components $x_i$ of $x$ yields  the same discussion.
As the discussion in Theorem \ref{theorem5.2}, it follows that
\begin{align}
&D_{KL}(\tilde{\mu}^y_{NJ}||\tilde{\mu}^y_{N})\leq D_{KL}(\tilde{\mu}^y_{NJ}||\tilde{\mu}^y_{N})+D_{KL}(\tilde{\mu}^y_{N}||\tilde{\mu}^y_{NJ})\nonumber\\
&\leq \frac{1}{\gamma_N}\int f_{\Lambda_N}(\mu_J-\mu)\log\frac{\mu_J}{\mu}=\frac{1}{\gamma_N}\int f_{\Lambda_N}(\mu_J-\mu)\log\frac{f^{(q)}_J(x)}{f^{(q)}(x)}\\
&=\frac{1}{\gamma_N}\int f_{\Lambda_N}f^{(q)}(x)dx\frac{f^{(q)}_J(x)-f^{(q)}(x)}{f^{(q)}(x)}\log\frac{f^{(q)}_J(x)}{f^{(q)}(x)}.\nonumber
\end{align}
Denoting $\frac{f^{(q)}_J(x)-f^{(q)}(x)}{f^{(q)}(x)}:=u$, we write
\begin{align}
D(f_J, f):=\frac{f^{(q)}_J(x)-f^{(q)}(x)}{f^{(q)}(x)}\log\frac{f^{(q)}_J(x)}{f^{(q)}(x)}=u\log(1+u).
\end{align}
By the expressions of $f^{(q)}_J(x)$ and $f^{(q)}(x)$, the common factor
$\sqrt{4-(1-q)x^2}$ contains two null points, which are  all the zero points of $f^{(q)}_J(x)$ and $f^{(q)}(x)$. And the remainder factors of them have no zero points, which yields according to Lemma \ref{lemma2.1}
\begin{align}
u=\frac{
\sum\limits_{k=J}^{\infty}
(-1)^{k-1}q^{\left(\begin{array}{c}k \\2\end{array}\right)}T_{2k-2}
(\frac{x\sqrt{1-q}}{2})}{
\sum\limits_{k=1}^{\infty}
(-1)^{k-1}q^{\left(\begin{array}{c}k \\2\end{array}\right)}T_{2k-2}
(\frac{x\sqrt{1-q}}{2})}\rightarrow 0 \,\, \text{uniformly as}\,\, J\rightarrow \infty.
\end{align}
Thus $D(f_J, f)=O(u^2)\rightarrow 0$ as $J\rightarrow\infty$.
This gives our conclusion.
\end{proof}
 By using the inequalities \eqref{errr5.5}, \eqref{errr5.6}, it follows from  Theorem \ref{theorem5.2}, \ref{theorem5.3}  that
\begin{corollary}
The same conditions as in Theorem \ref{theorem5.2} hold. We have for $N\rightarrow\infty, J\rightarrow\infty$
\begin{align}
D_{\text{Hell}}(\tilde{\mu}^y_{NJ}, \mu^y)\rightarrow 0, \\
D_{\text{TV}}(\tilde{\mu}^y_{NJ}, \mu^y)\rightarrow 0.
\end{align}
\end{corollary}

\section{Numerical test}
In the numerical examples, we implement the Metropolis-Hastings algorithm  \cite{stuart} to sample from the posterior distribution. This algorithm aims at sampling from a target distribution $\mu(dx)$ with density $\pi(x)$. Algorithm MH provides the details of this algorithm.

\begin{tabular}{l}
\hline
{\bf Algorithm MH:} Metropolis-Hastings algorithm\\
\hline
Initialize $x^{(0)}\in\mathcal{X}$ \\
for $i=1, 2, \cdots$ do\\
\hspace{0.5cm} Propose: Move $x^{(i-1)}$ to a candidate $\hat{x}$ \\
\hspace{2cm} according to a transition density $q(x|x^{(i-1)})$.
\\
\hspace{0.5cm} Acceptance probability:
\\
\hspace{1cm}$\alpha(\hat{x}|x^{(i-1)})=\min\{1, \frac{q(x^{(i-1)}|\hat{x})\pi(\hat{x})}{q(\hat{x}|x^{(i-1)})\pi(x^{(i-1)})}\}$.
\\
\hspace{0.9cm} $u\sim$ Uniform$(0, 1)$.
\\
\hspace{0.5cm} Accept the proposal $\hat{x}$ with probability $\alpha$, i.e.,
\\
\hspace{1cm} $
x^{(i)}=\left\{
\begin{aligned}
&\hat{x}, & \text{if}\,\,  \alpha>u,\\
&x^{(i-1)}, & \text{otherwise}.
\end{aligned}
\right.
$\\
end for
\\
\hline
\end{tabular}

\subsection{1D problem}
We consider the estimation problem of the unknown mean $x$ according to random realizations $\{y_i\}_{i=1}^n$, which are drawn independently from a Gaussian distribution $N(y|x, \sigma^2)$ with a given standard deviation $\sigma$. It can be seen that if the prior is taken as Gaussian distribution, the posterior is also Gaussian.
Here we suppose that $x$ obeys a q-Gaussian prior distribution. The posterior distribution exhibits non-Gaussianity.  We use this simple example as the first testbed.
The likelihood function can be written as
\begin{align*}
f(y| x)&=\prod_{i=1}^n \frac{1}{(2\pi \sigma^2)^{\frac{1}{2}}}\exp(-\frac{(y_i-x)^2}{2\sigma^2})\\
&=
 \frac{1}{(2\pi \sigma^2)^{\frac{n}{2}}}\exp(-\frac{\sum_{i=1}^n(y_i-x)^2}{2\sigma^2}).
\end{align*}
Thereby, we get the posterior according to Bayes' formula
\begin{align*}
f^y(x)=\frac{f(y| x)f^{(q)}(x)}{\int f(y| x)f^{(q)}(x)dx}.
\end{align*}
We generate the data $y_i$ from Gaussian distribution $N(y|x^\dag, \sigma^2)$ with the true mean $x^\dag=10$ and the standard deviation $\sigma=5$. For the numerical experiment, the pseudo-random numbers $$y=[15.0389;-0.6183;7.4771;3.6470;8.0871;13.2434;14.1286;4.9253;7.6447;10.6851]$$ are used as synthetic data. In the Christoffel algorithm, the density function for the pluripotential equilibrium measure $\nu$ is taken as
\begin{align}
v(x)=\frac{1}{\sqrt{\Xi}}\frac{\sqrt{1-q}}{\pi\sqrt{4-(1-q)\frac{(x-x_0)^2}{\Xi}}}.
\end{align}
The parameter $x_0$ and $\Xi$ are specified as $x_0=11.5$ and $\Xi=c(x_0-x^\dag)^2(1-q)/4$ with a fixed positive constant $c>1$. In the numerical tests, the prior density is truncated to the first $100$ terms, i.e., $J=100$. The normalization constants $\gamma_{NJ}$ and $\gamma_J$ (has a similar definition with $\gamma$) are computed by Monte Carlo method.
We list the relative error of the posterior density function in Table \ref{tab:tab1}. The results show that the spectral likelihood expansion (SLE) approximation can fit the likelihood function well and the numerical precision is higher with higher order polynomial. For comparison purposes, we draw some samples from the posterior distribution by means of MCMC. An independence sampling process is utilized. The likelihood function and its SLE approximation, the posterior density and a normalized histogram of the obtained random walk samples are shown in Fig. \ref{ex1_1}. From the results, it can be seen that the SLEs are able to approximate the likelihood function well. As depicted in \cite{nagel}, these well-fitted regions accumulate the largest proportions of the total prior probability mass. Due to the reason that the q-Gaussian priors have compact, even though the SLEs start strongly deviating from the likelihood function, the posterior densities by using SLEs vanish when far away from the peaks. From the viewpoint, the proposed priors are better options for fitting the posterior densities by using SLEs. For the sampling effect, the histogram plots can reflect the posterior densities well.

\begin{table}[h]
\centering
\caption{The relative errors for posterior density with different $q$ and orders $N$.}\label{tab:tab1}
\begin{tabular}{|c|c|c|c|c|c|}
\hline
$q$ & $N=2$ & $N=5$&$N=7$&$N=9$& $N=12$ \\
\hline
-0.5 & 0.0772&0.0065&0.0044&1.8092e-04&6.8778e-05\\
\hline
-0.2 & 0.0963&0.0077&0.0048&1.9242e-04&9.2320e-05\\
\hline
0 & 0.1083&0.0088&0.0061&1.9865e-04&5.1662e-05\\
\hline
0.2&0.1202&0.0097&0.0069&2.3369e-04&6.0067e-05
\\
\hline
0.5&0.1627&	0.0120	&0.0085	&2.6809e-04&	7.0980e-05
\\
\hline
\end{tabular}
\end{table}

%
\begin{figure}[!hbt]
\centering
\subfigure[$q=-0.8$]{
\includegraphics[width=3.5cm]{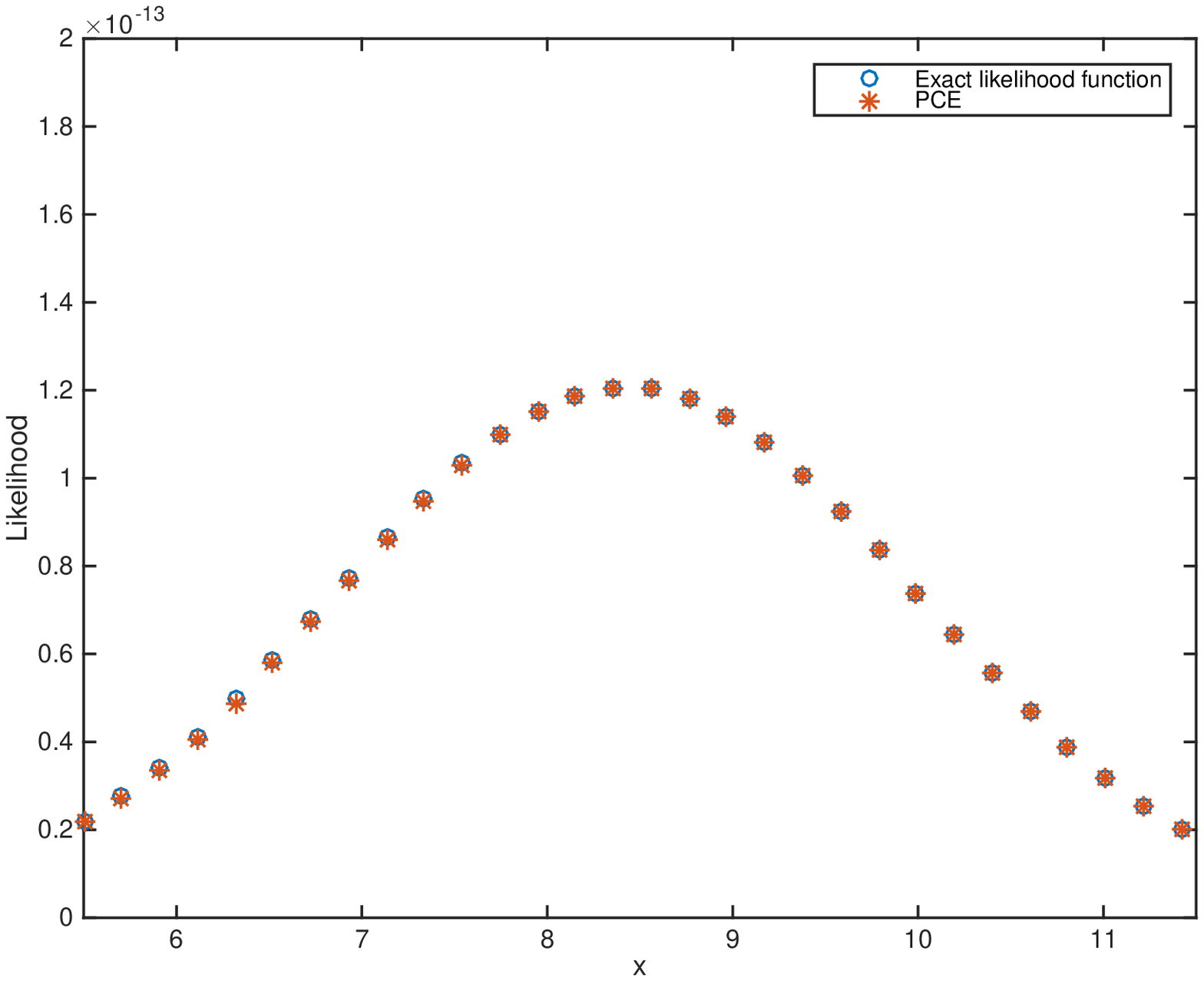}
\includegraphics[width=3.5cm]{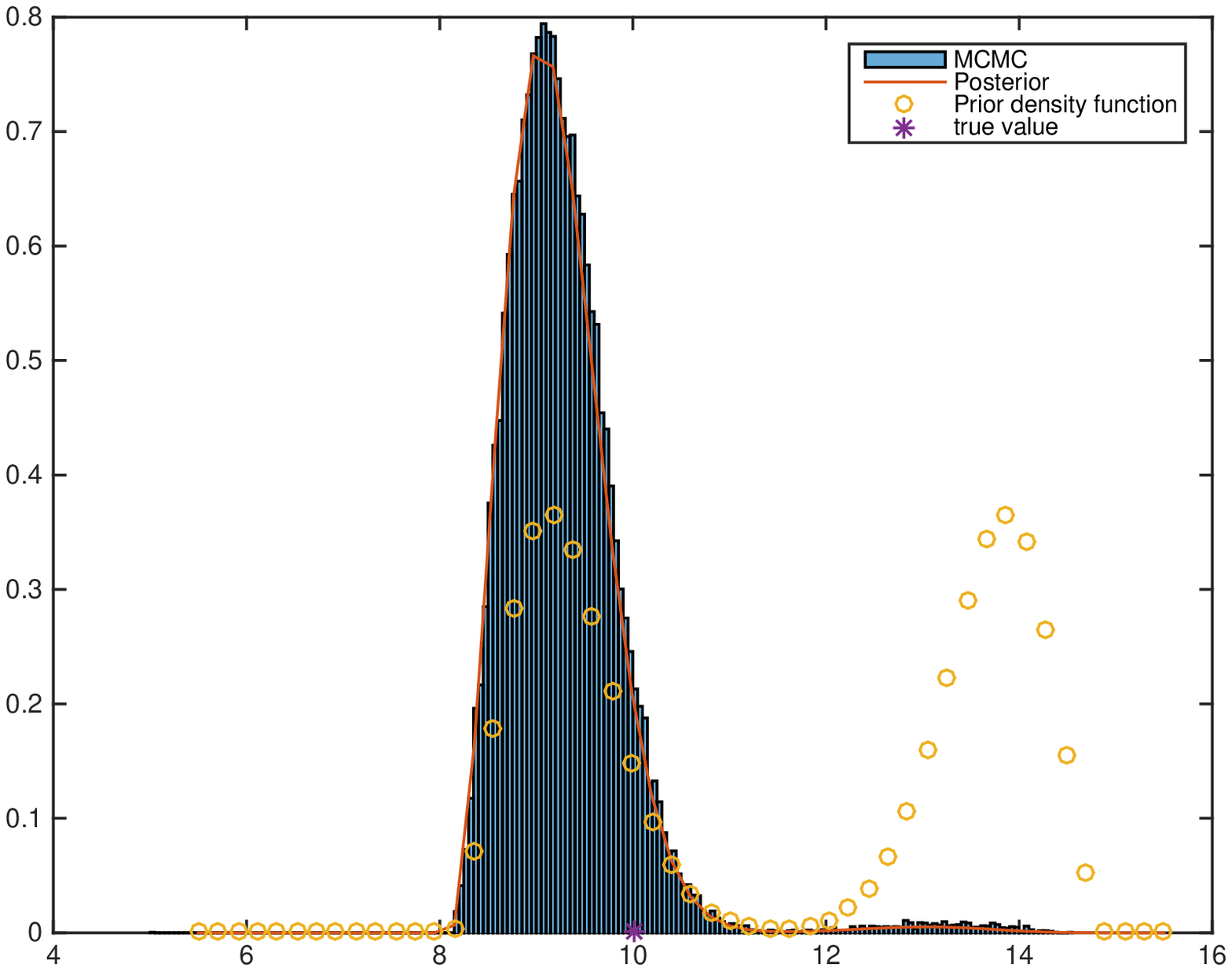}
\includegraphics[width=3.5cm]{pde_n8_19.eps}}
\subfigure[$q=-0.5$]{
\includegraphics[width=3.5cm]{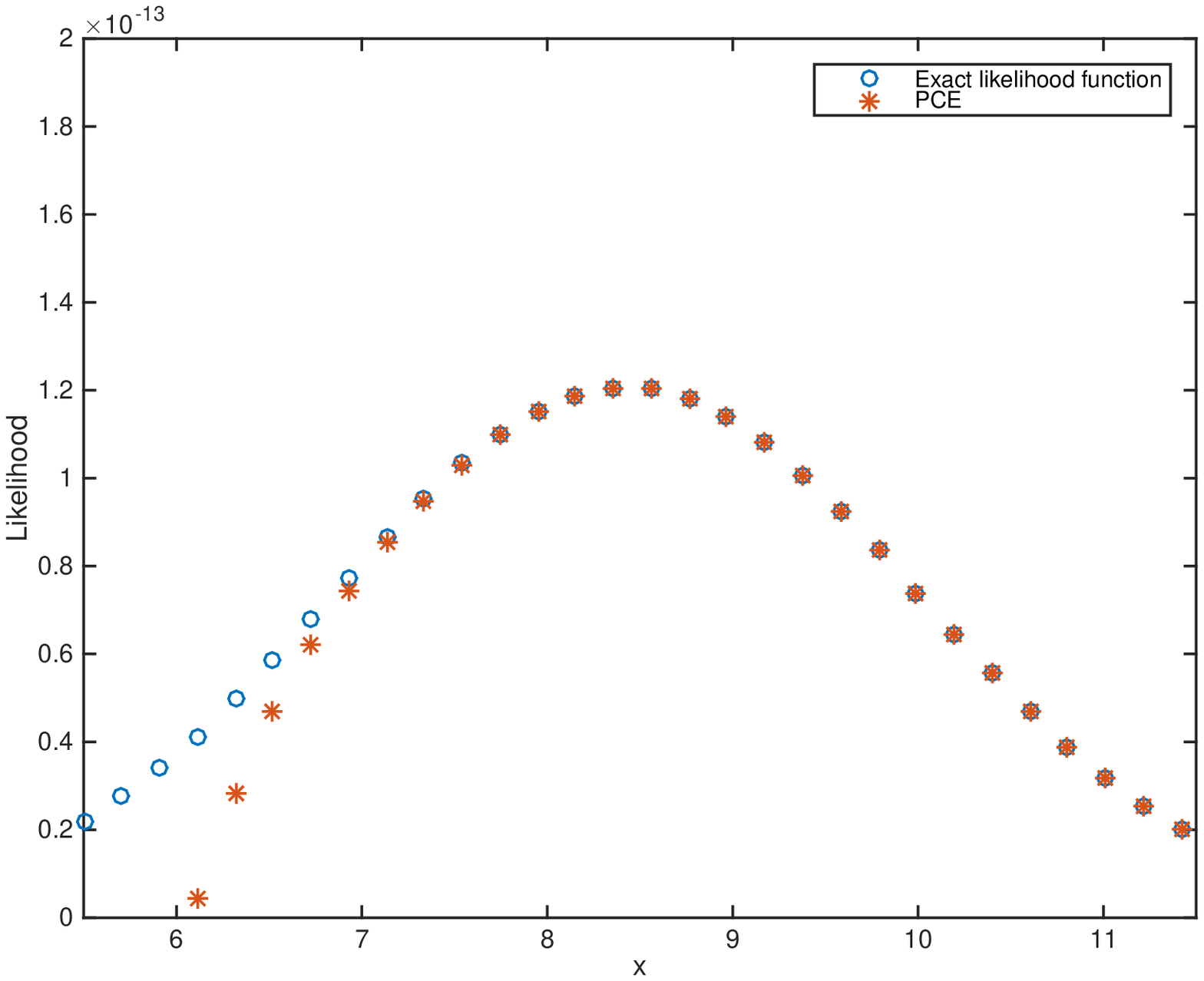}
\includegraphics[width=3.5cm]{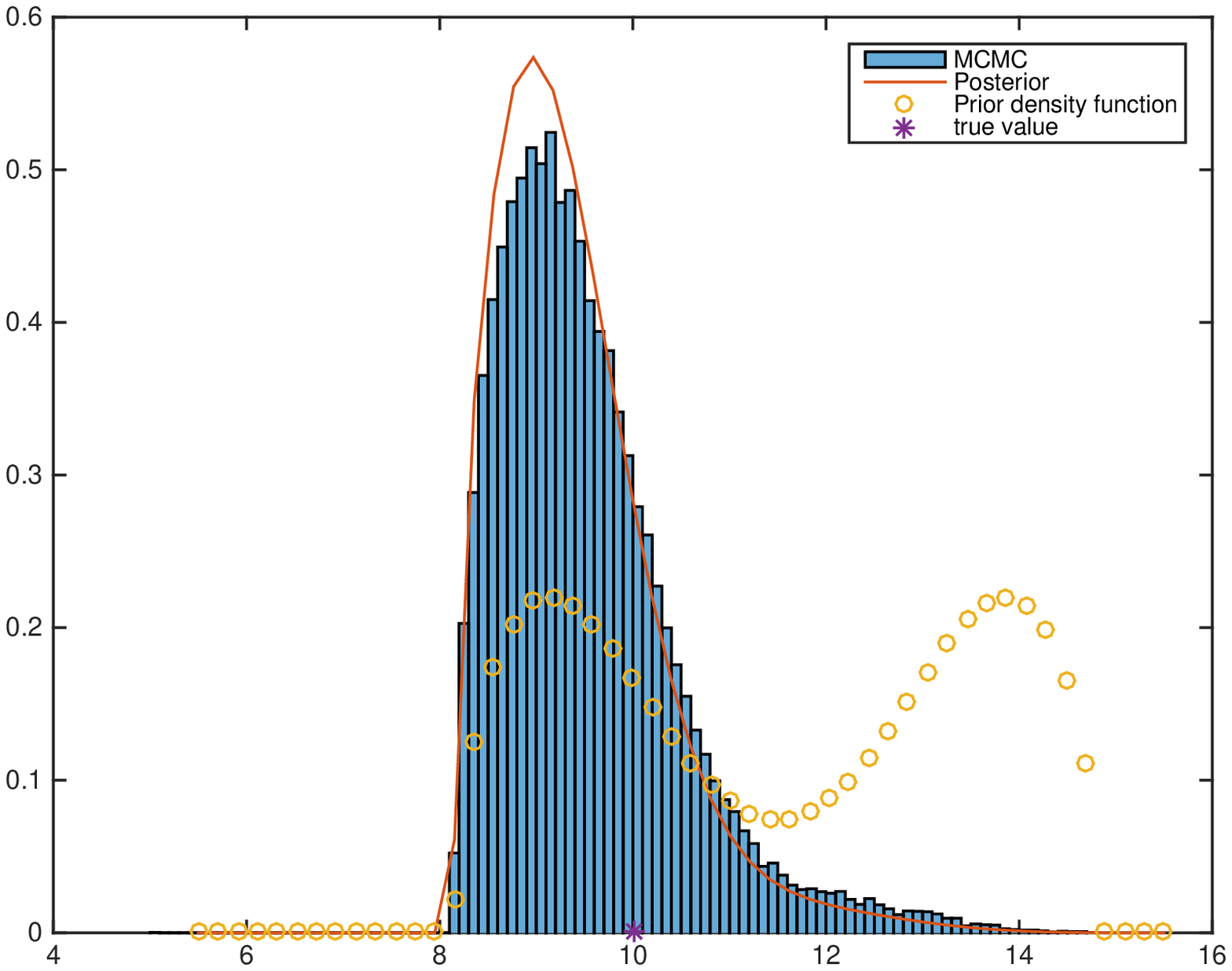}
\includegraphics[width=3cm]{pde_n5.eps}}
\subfigure[$q=-0.2$]{
\includegraphics[width=3.5cm]{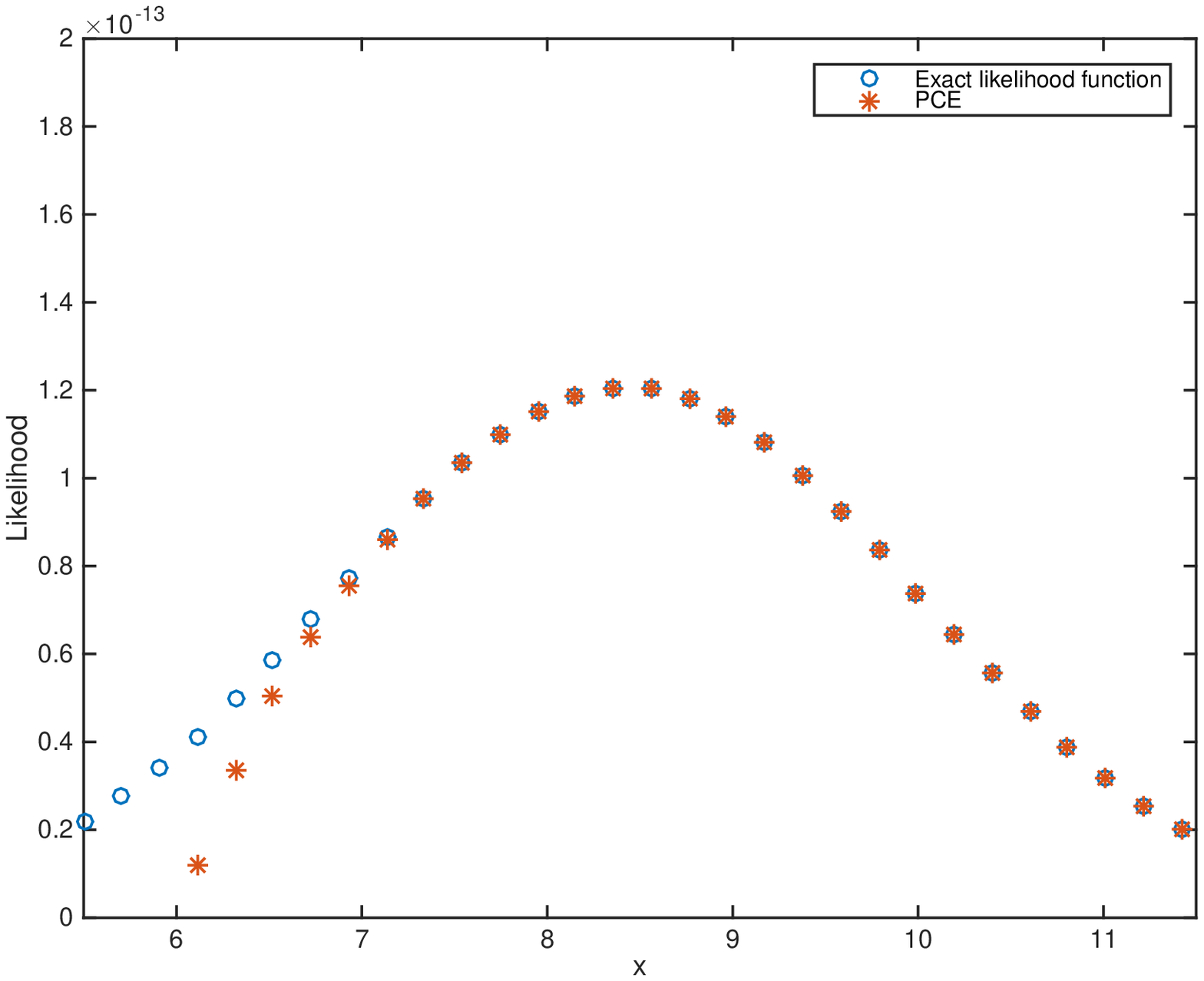}
\includegraphics[width=3.5cm]{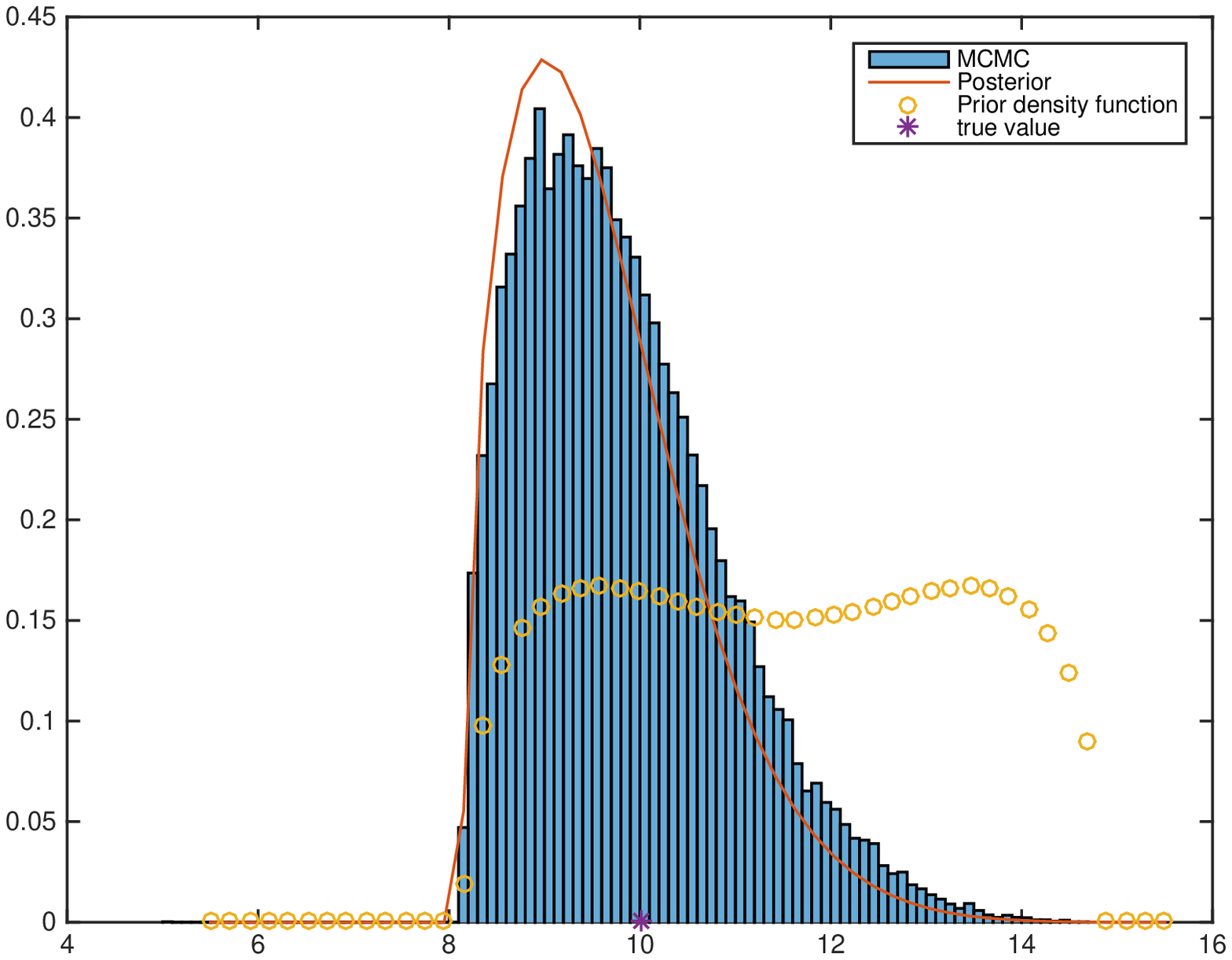}
\includegraphics[width=3.5cm]{pde_n2.eps}}
\subfigure[$q=0$]{
\includegraphics[width=3.5cm]{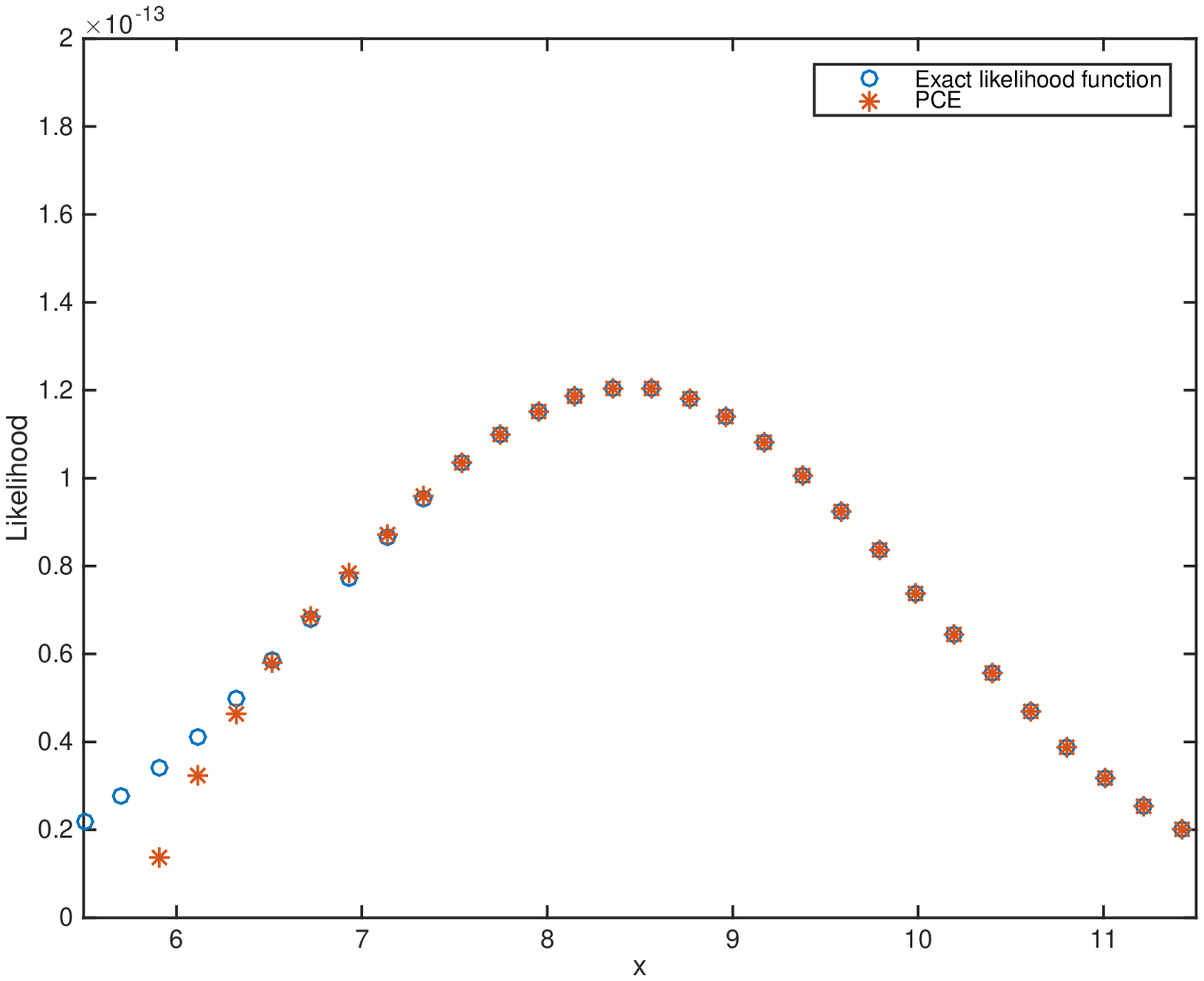}
\includegraphics[width=3.5cm]{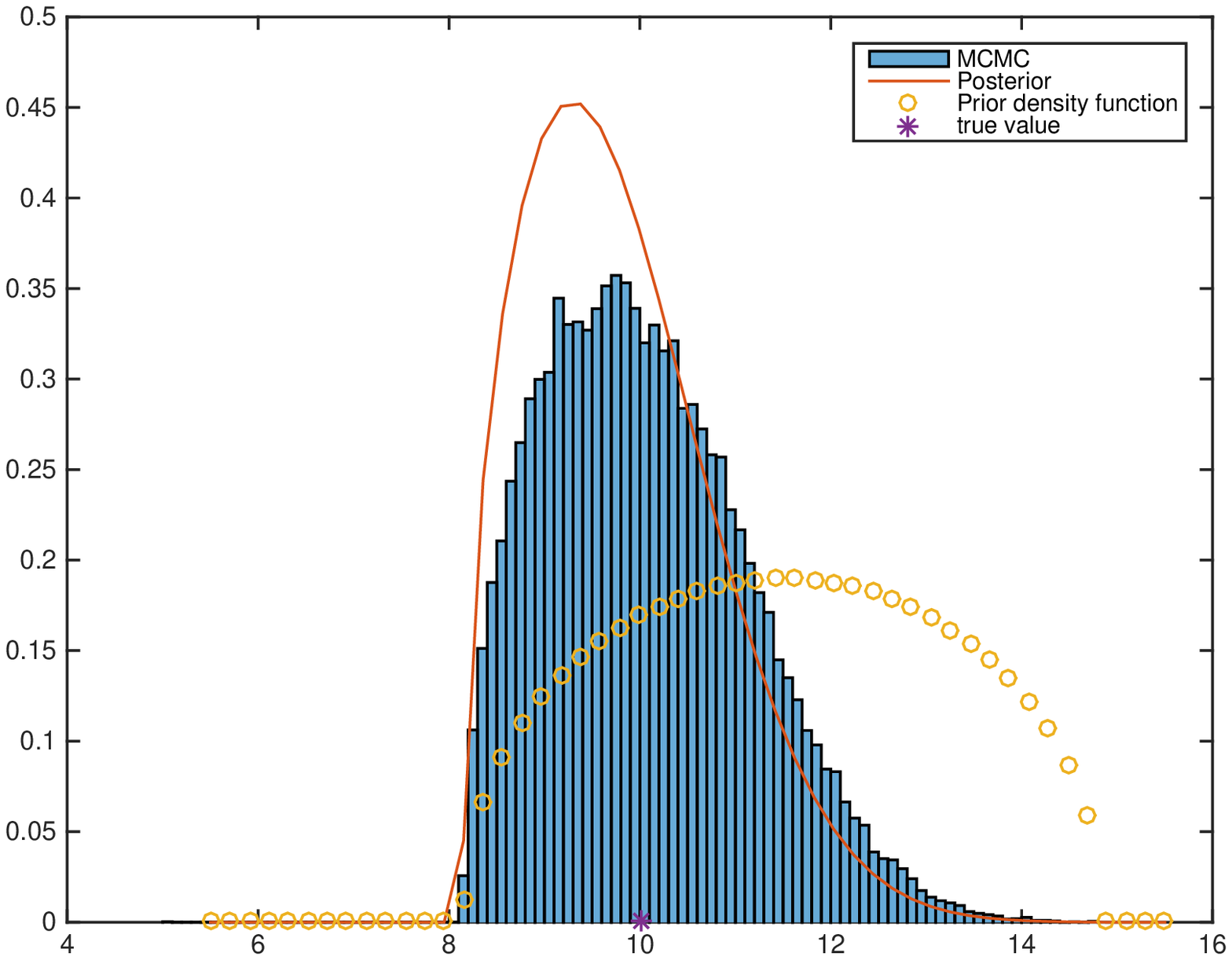}
\includegraphics[width=3.5cm]{pde_0.eps}}
\subfigure[$q=0.2$]{
\includegraphics[width=3.5cm]{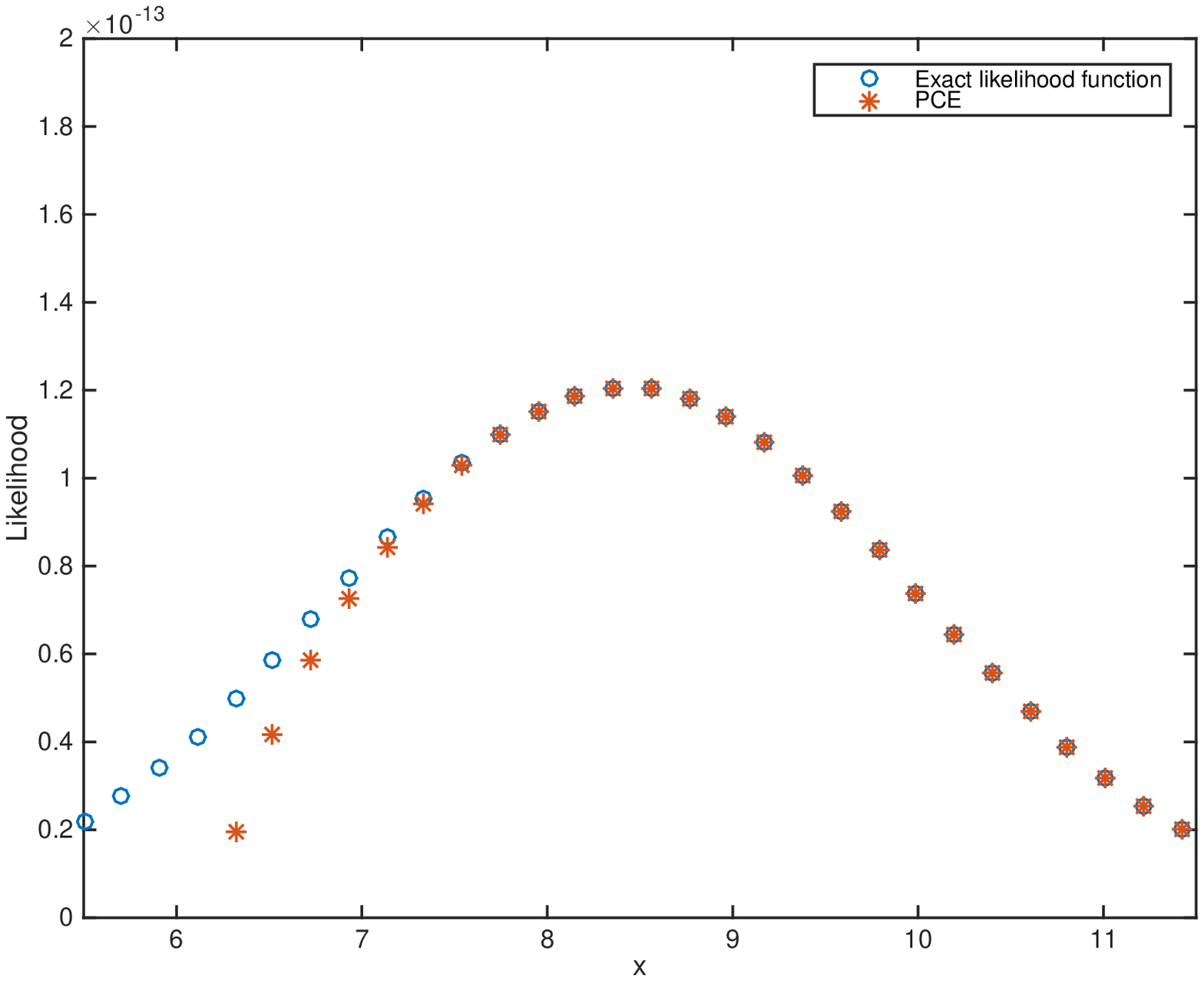}
\includegraphics[width=3.5cm]{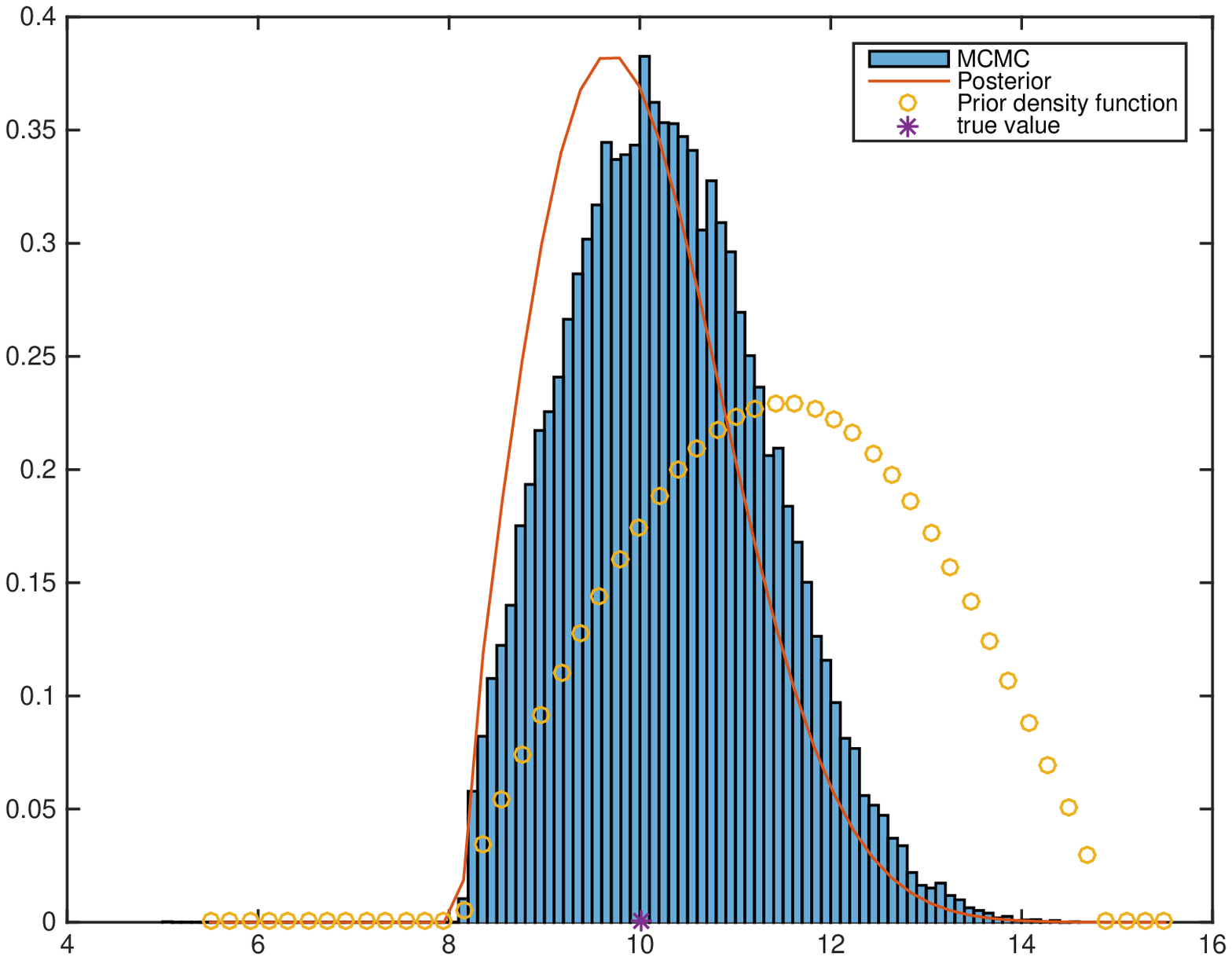}
\includegraphics[width=3.5cm]{pde_p2.eps}}
\subfigure[$q=0.5$]{
\includegraphics[width=3.5cm]{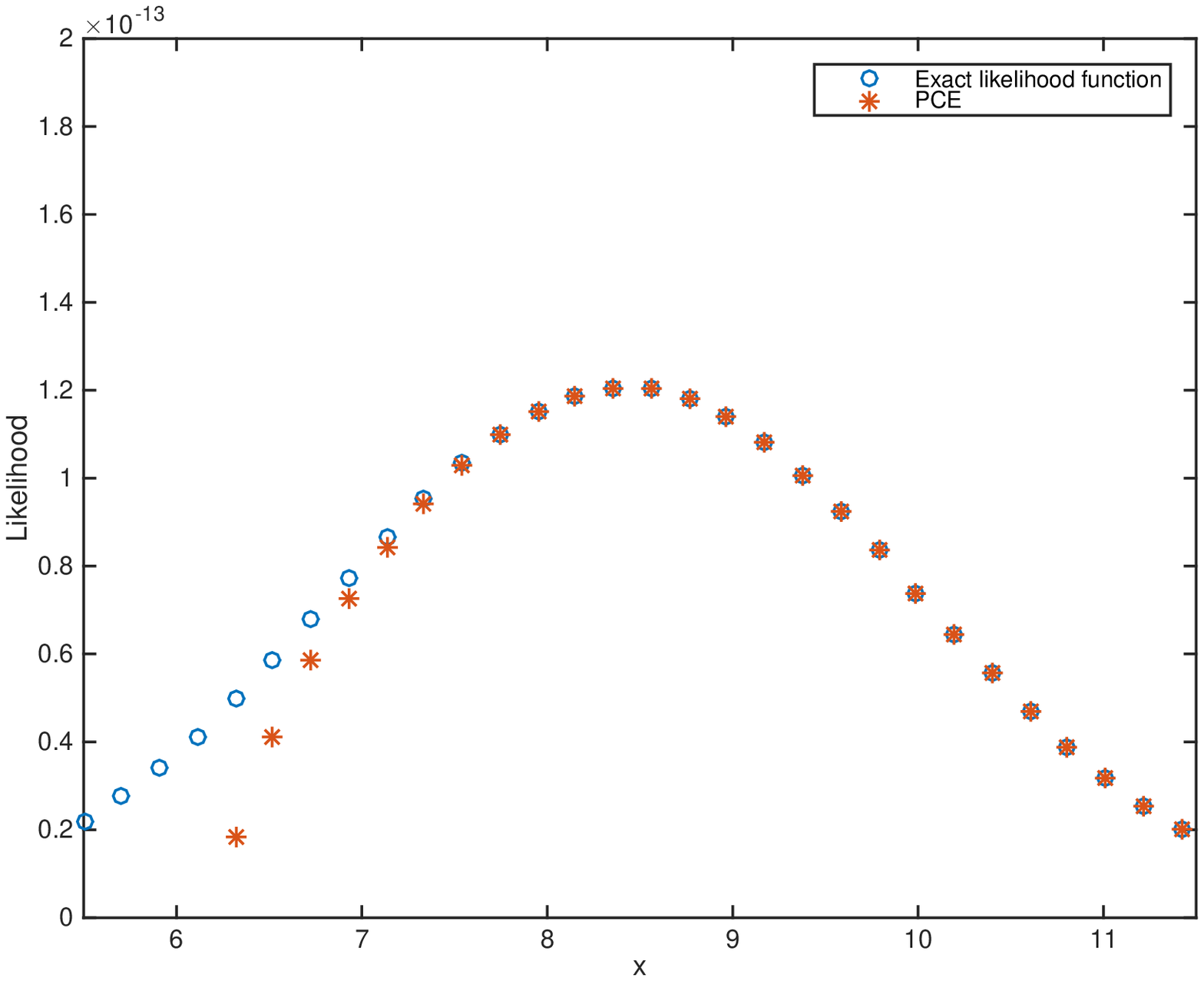}
\includegraphics[width=3.5cm]{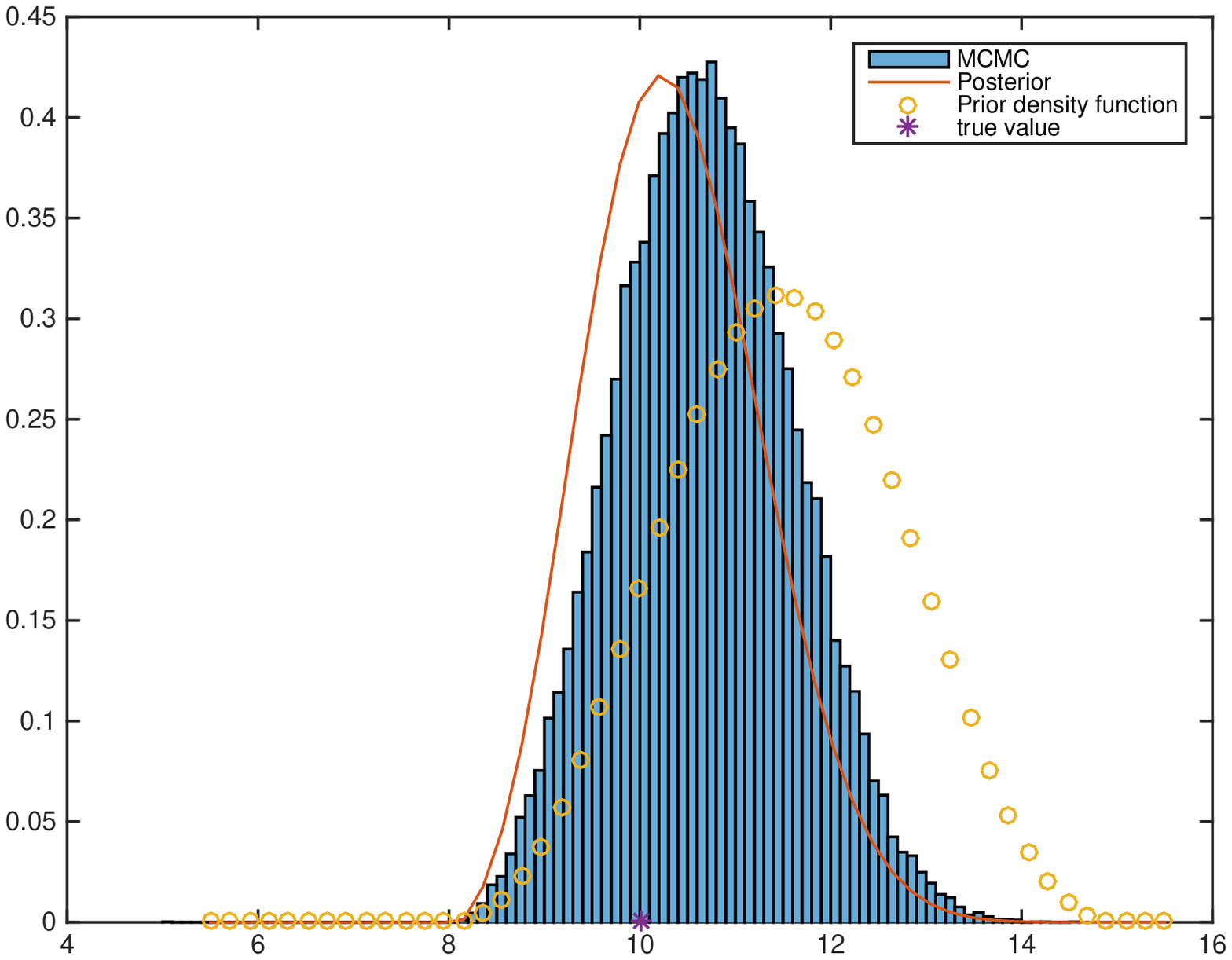}
\includegraphics[width=3.5cm]{pde_p5.eps}}
\caption{Left: Likelihood function and its PCE approximation with $N=12$;  Middle: original model; Right: SLE model}
\label{ex1_1} 
\end{figure}

\subsection{2D problem}
We consider a stationary inverse heat conduction problem governed by  
\begin{align}\label{id5.1}
&-\nabla\cdot(\kappa\nabla u)=0\,\, \text{in}\,\,\Omega\in\mathbb{R}^2,\\
&u\mid_{\gamma_1}=g(x, y),\label{id5.2}\\
&-\kappa_0\frac{\partial u}{\partial\nu}\mid_{\gamma_2}=h(x, y),\label{id5.3}
\end{align}
where $\Omega=\cup_{i=0}^n\Omega_i$ with boundary $\partial\Omega=\gamma_0\cup\gamma_1$, $\kappa$ takes different values $\kappa_i$ at $\Omega_i$ and $g, h$ are the given functions. For the test example, we use the same settings as in \cite{nagel}, where $\Omega$ is a square domain $(x, y)\in (0, 1)\times(0, 0.6)$ and $\kappa_0=15, \kappa_1=32, \kappa_2=28$. The subdomains $\Omega_1, \Omega_2$ are disks located at $(0.3, 0.3), (0.7, 0.3)$ with radius $0.1$ respectively.  The Dirichlet boundary condition $g(x, 0.6)=200$ and Neumann boundary conditions $h(0, y)=0, h(1, y)=0, h(x, 0)=2000$. The forward problem is to find the solution $u$ for given $\kappa_1, \kappa_2$. The inverse heat conduction problem is to seek $\kappa_1, \kappa_2$ by measurements of $u$ at some fixed points in domain $\Omega$. The domain settings are displayed in Fig. \ref{fig:1}.
The forward problem is solved in linear finite element method and we show the finite element solution in Fig. \ref{fig:2}. The data is acquired at $10$ scattered points distributed in the domain $\Omega$ by adding absolute error to the numerical solution of the forward problem, i.e.,
\begin{align*}
u(\vec{x}, \vec{y})=u^\dag(\vec{x}, \vec{y})+\delta*\text{randn}(10, 1),
\end{align*}
where $(\vec{x}, \vec{y})$ is the 10 measure point position vector. We display several numerical effects in Fig. \ref{fig:3} for different $q$ and noise $\delta$. The results show that when $\delta$ is smaller, the peak of posterior density is closer to the true value. And the MCMC samples can reflect the posterior density well. However, there exists some differences  away from the peak between the posterior density with the SLE and the original model. This also leads to some samples do not concentrate close to the peak. 

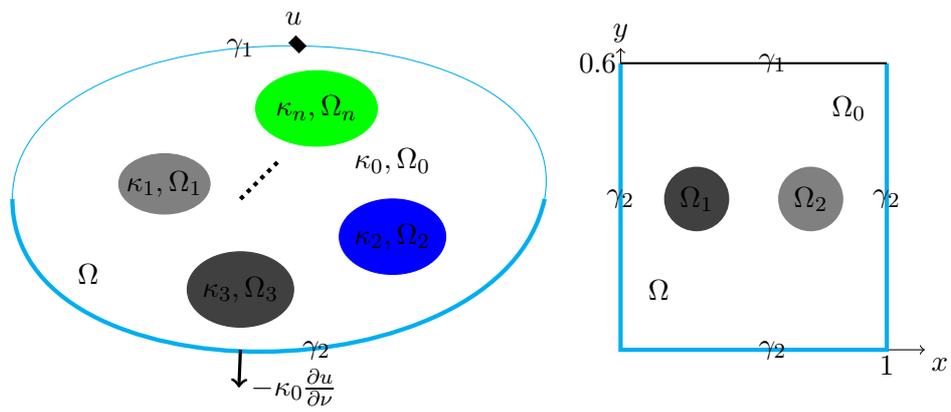
\begin{figure}[!hbt]
\begin{tikzpicture}[xscale=1,yscale=1]
\draw [black, cyan] (0,0) to [out=90,in=80] (7,0);
\draw [black,ultra thick, cyan]  (7,0) to [out=-102,in=-90] (0,0);
\draw [fill, color=gray] (2,0.2) ellipse [x radius=0.6cm, y radius=0.4cm];
\draw [fill, color=blue] (5,-0.5) ellipse [x radius=0.7cm, y radius=0.5cm];
\draw [fill, color=green] (4,1.2) ellipse [x radius=0.8cm, y radius=0.5cm];
\draw [fill, color=darkgray] (3,-1.2) ellipse [x radius=0.7cm, y radius=0.5cm];
\node at (5,0.5) {$\kappa_0, \Omega_0$};
\node at (2,0.2) {$\kappa_1, \Omega_1$};
\node at (5,-0.5) {$\kappa_2, \Omega_2$};
\node at (3,-1.2) {$\kappa_3, \Omega_3$};
\node at (4,1.2) {$\kappa_n, \Omega_n$};
\draw[black, very thick,->] (3,-2)--(2.98,-2.5);
\node at (3.7,-2.5) {$-\kappa_0\frac{\partial u}{\partial \nu}$};
\node at (3.7,2.4) {$u$};
\draw [line width=0.2cm] (3.7,2) -- (3.8,2.1);
\draw [dotted, ultra thick] (3,0) -- (3.5,0.5);
\node at (3, 2) {$\gamma_1$};
\node at (4, -2) {$\gamma_2$};
\node at (1, -1) {$\Omega$};
\draw [<->] (12,-2) -- (8,-2) -- (8,2);
\draw [ultra thick, cyan] (8, 1.8) -- (8, -2) -- (11.5,-2) -- (11.5, 1.8); 
\draw [thick] (8, 1.8) -- (11.5, 1.8);
\draw [darkgray, ultra thick,fill] (9,0) circle [radius=0.4];
\draw [gray, ultra thick,fill] (10.5,0) circle [radius=0.4];
\node at (8.5, -1.2) {$\Omega$};
\node at (11, 1.2) {$\Omega_0$};
\node at (9, 0) {$\Omega_1$};
\node at (10.5, 0) {$\Omega_2$};
\node at (10,-2) {$\gamma_2$};
\node at (11.5,0) {$\gamma_2$};
\node at (8,0) {$\gamma_2$};
\node at (10,1.8) {$\gamma_1$};
\node at (12.2, -2.2) {$x$};
\node at (11.5, -2.2) {$1$};
\node at (8, 2.2) {$y$};
\node at (7.7, 1.8) {$0.6$};
\end{tikzpicture}
\caption{Heat conduction setup.}
\label{fig:1}
\end{figure}

\begin{figure}[!hbt]
\centering
\includegraphics[width=6cm]{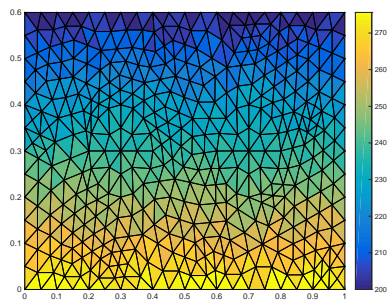}
\caption{The finite element solution.}
\label{fig:2} 
\end{figure}

\begin{figure}[!hbt]
\centering
\subfigure[$q=-0.8, \delta=0.1$]{
\includegraphics[width=6cm]{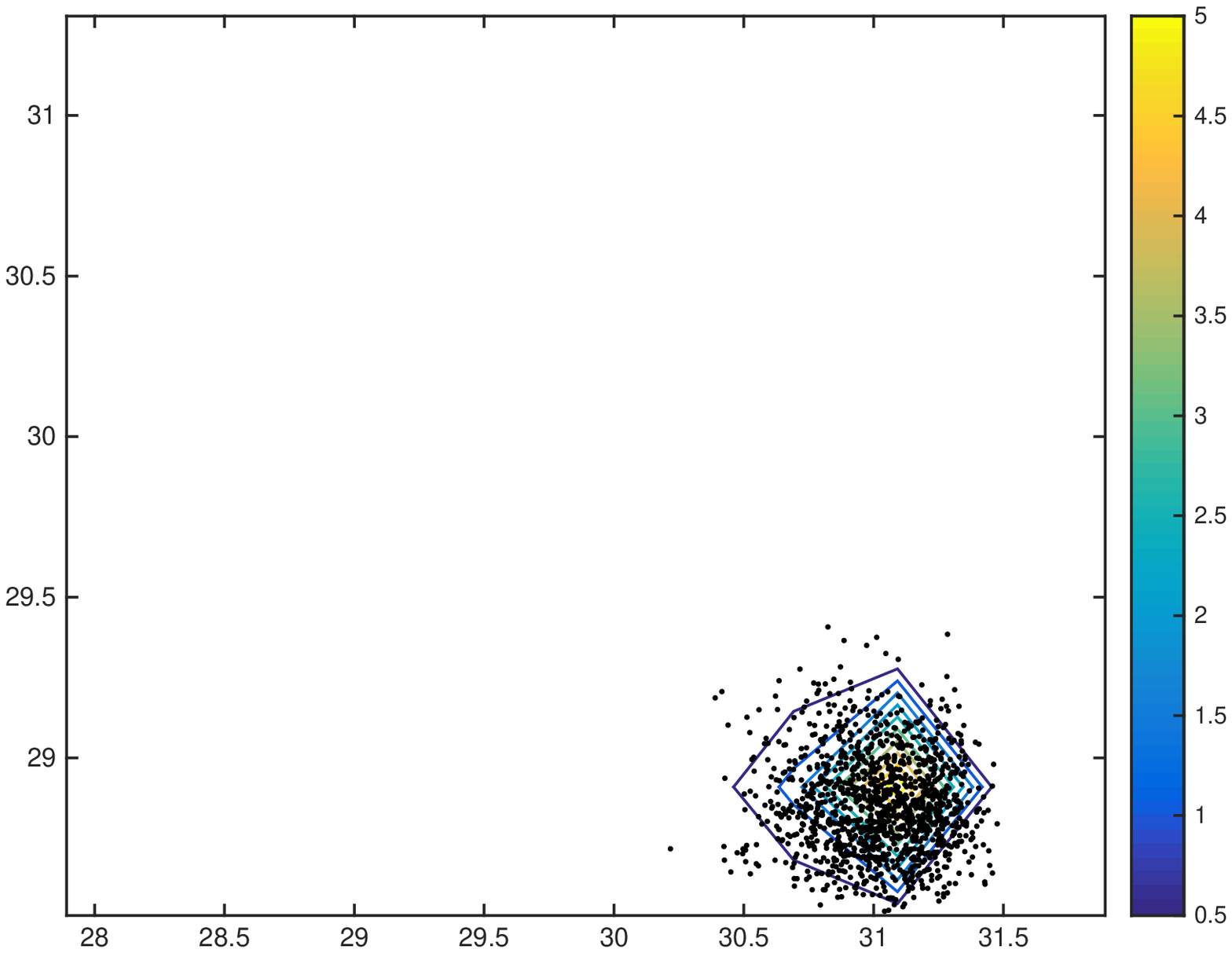}
\includegraphics[width=6cm]{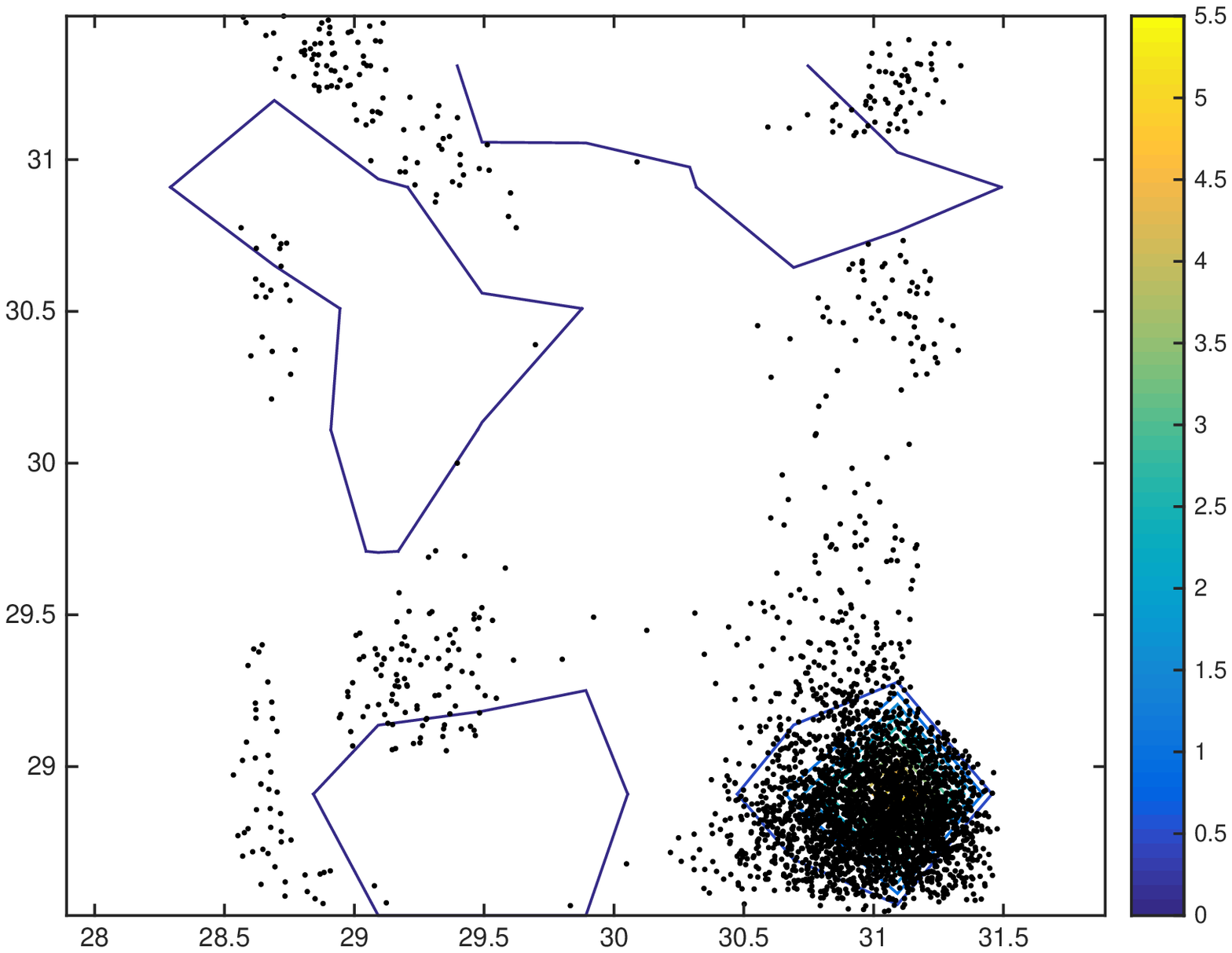}}
\subfigure[$q=0.5, \delta=1$]{
\includegraphics[width=6cm]{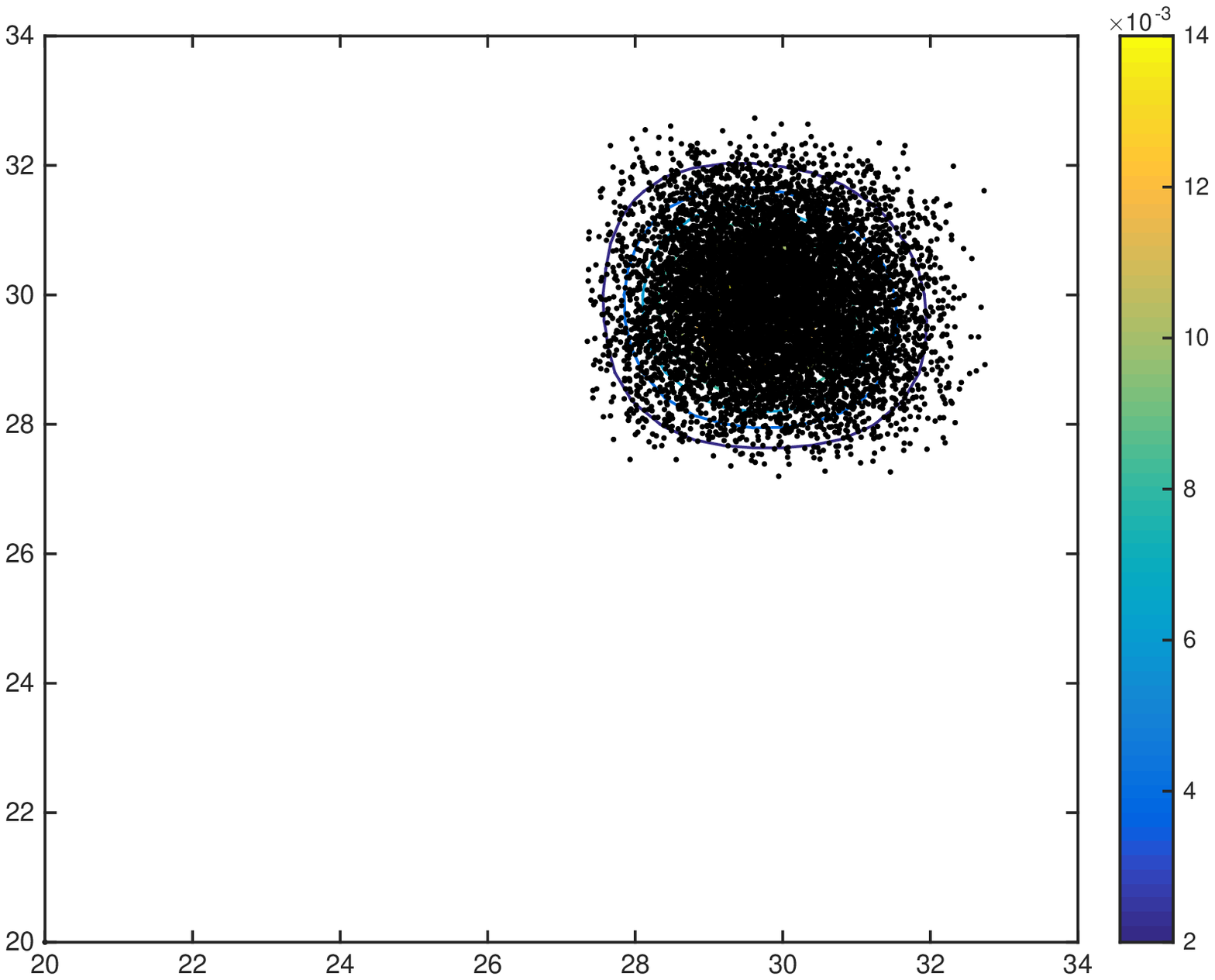}
\includegraphics[width=6cm]{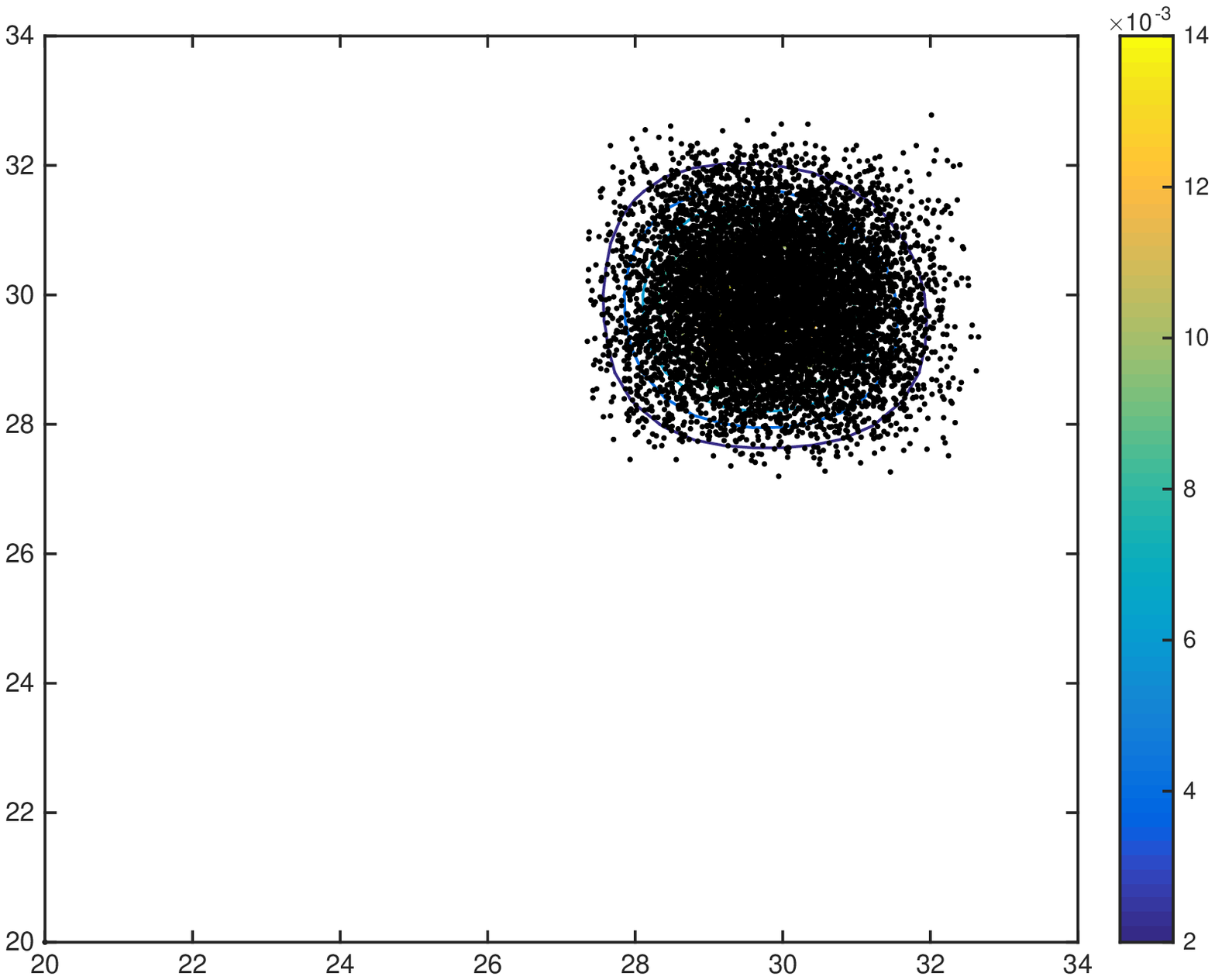}}
\subfigure[$q=0.5, \delta=0.5$]{
\includegraphics[width=6cm]{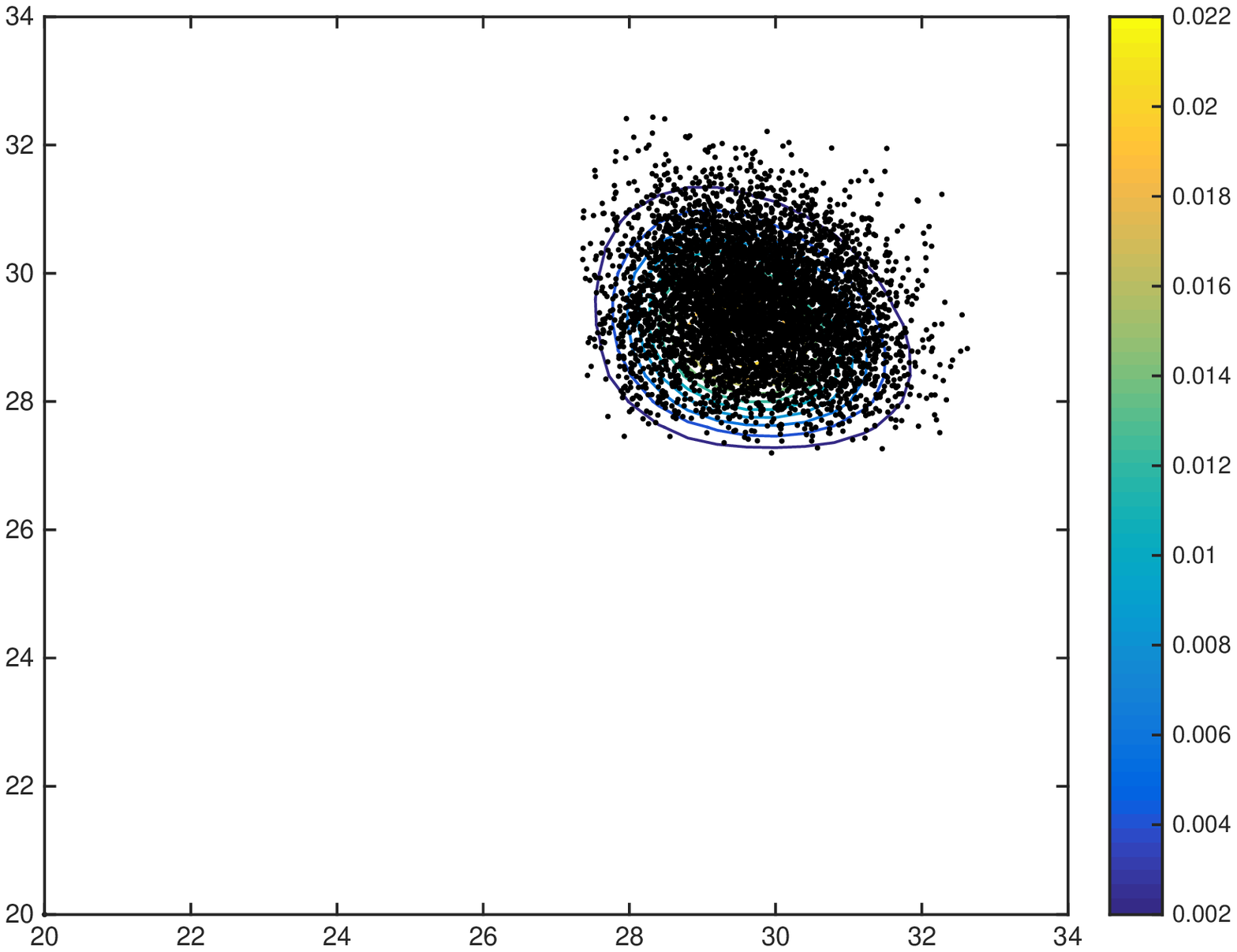}
\includegraphics[width=6cm]{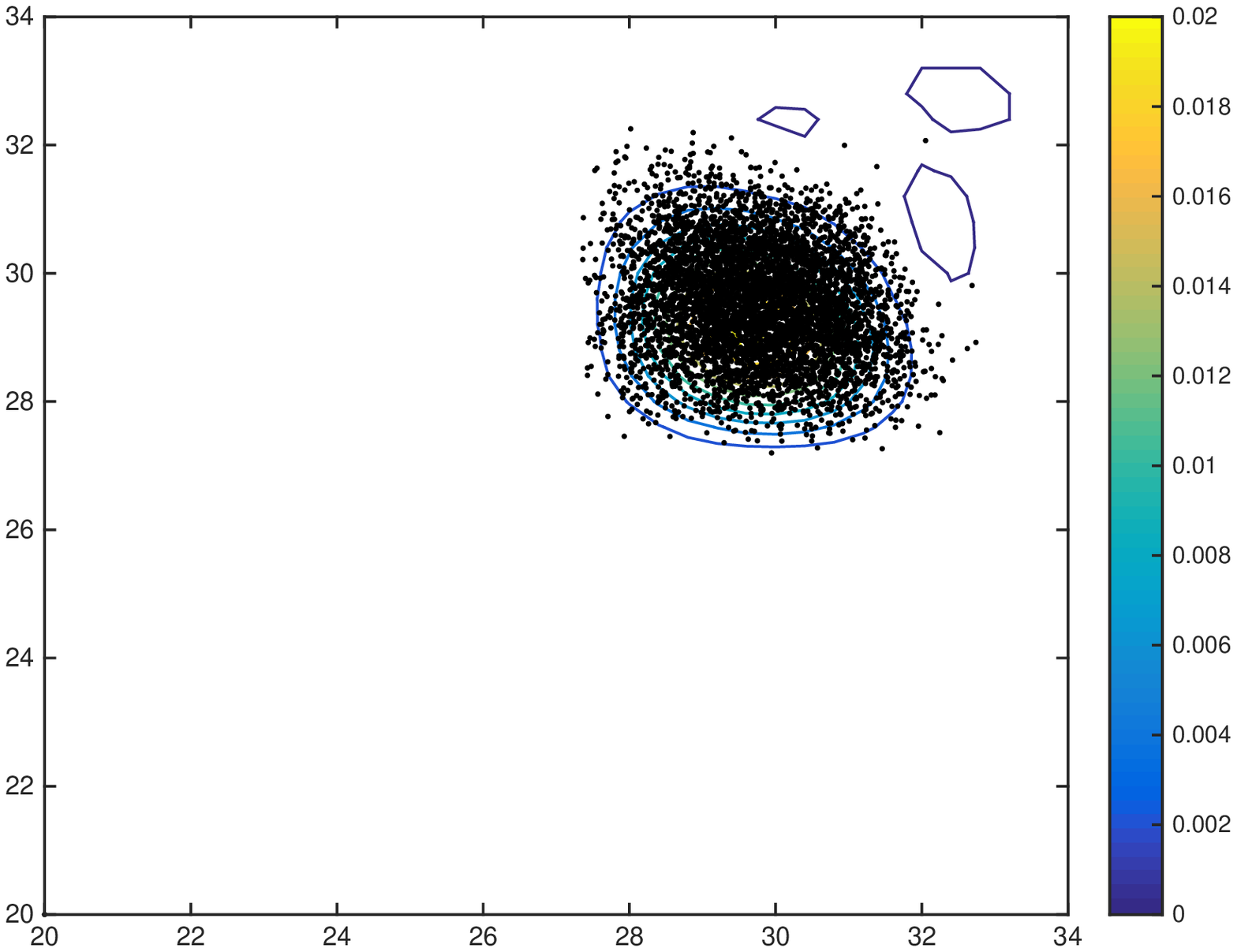}}
\caption{The contour of posterior density and MCMC samples. Left: Original model; Right: The SLE model}
\label{fig:3} 
\end{figure}




 \section*{Conclusions}


We have introduced a new prior model, namely,  q-Gaussian prior, 
into the research of inverse problems, which is q-analogue of classical Gaussian distribution.  Since the density function of each q-Gaussian distribution has compact support, we can characterize some bounded physical parameters using it. In order to accelerate the computation of MCMC sampling in Bayesian inversion, we adopted a spectral likelihood approximation algorithm based on q-Hermite polynomial chaos expansion of likelihood function. Then we proved the convergence of posterior distribution in the framework of relative entropy when the likelihood function is replaced with truncated PCE. And when the q-Gaussian prior is approximated, we also studied the convergence of the corresponding approximated posterior measure. With the proposed prior and SLE algorithm, we  verified the effectiveness of the proposed method  through two numerical examples . 


\end{document}